\documentclass[11pt]{amsart}
\usepackage{amsmath}
\usepackage{amsthm}
\usepackage[rgb]{xcolor}
\usepackage{hyperref,color}
\definecolor{dred}{rgb}{0.4,0.0,0.0}
\definecolor{dgrn}{rgb}{0.0,0.4,0.2}
\definecolor{dblu}{rgb}{0.0,0.0,0.4}
\definecolor{dpur}{rgb}{0.7,0.0,0.7}
\usepackage[margin=1.0in]{geometry}
\usepackage{graphicx}
\usepackage{bm}

\numberwithin{equation}{section}

\newtheorem{prop}{Proposition}
\newtheorem{lemma}[prop]{Lemma}

\newtheorem{thm}[prop]{Theorem}
\newtheorem{cor}[prop]{Corollary}

\numberwithin{prop}{section}

\theoremstyle{definition}
\newtheorem{defn}[prop]{Definition}

\newtheorem{rmk}[prop]{Remark}

\newcommand{\del}{\partial}
\newcommand{\dt}{\frac{\partial}{\partial t}}
\newcommand{\brs}[1]{\left| #1 \right|}
\newcommand{\ip}[1]{\left\langle #1 \right\rangle}
\newcommand{\brk}[1]{\left[ #1 \right]}
\newcommand{\prs}[1]{\left( #1 \right)}
\newcommand{\gG}{\Gamma}

\newcommand{\gz}{\zeta}
\newcommand{\gD}{\Delta}
\newcommand{\gd}{\delta}
\newcommand{\gs}{\sigma}

\newcommand{\gl}{\lambda}

\newcommand{\gw}{\omega}
\newcommand{\ga}{\alpha}
\newcommand{\gb}{\beta}
\renewcommand{\ge}{\epsilon}
\newcommand{\N}{\nabla}

\newcommand{\FF}{\mathcal F}

\newcommand{\til}[1]{\widetilde{#1}}

\renewcommand{\bar}[1]{\overline{#1}}

\newcommand{\bRm}{\mathbb{R}^m}
\newcommand{\lbr}{\left[}
\newcommand{\rbr}{\right]}

\newcommand{\hsp}{\hspace{1cm}}
\newcommand{\lap}{\Delta}
\newcommand{\la}{\lambda}

\newcommand{\bRn}{\mathbb{R}^n}

\newcommand{\hook}{\mathbin{\hbox{\vrule height2.4pt width4.5pt depth-2pt
\vrule height5pt width0.4pt depth-2pt}}}

\newcommand{\rN}{\mathring{\nabla}}
\newcommand{\bN}{\overline{\nabla}}
\newcommand{\ten}{\otimes}
\newcommand{\vp}{\varphi}
\newcommand{\lb}{\left[}
\newcommand{\rb}{\right]}
\newcommand{\calf}{\mathcal{F}}

\DeclareMathOperator{\tr}{tr}

\DeclareMathOperator{\Euc}{euc}

\begin{document}

\title[Entropy, stability and harmonic map flow]{Entropy, stability and
harmonic map flow}

\author{Jess Boling}
\email{\href{mailto:jboling@uci.edu}{jboling@uci.edu}}
\author{Casey Kelleher}
\email{\href{mailto:clkelleh@uci.edu}{clkelleh@uci.edu}}
\author{Jeffrey Streets}
\email{\href{mailto:jstreets@uci.edu}{jstreets@uci.edu}}

\address{Rowland Hall\\
         University of California\\
         Irvine, CA 92617}

\thanks{The second author was supported by an NSF Graduate Research Fellowship
DGE-1321846.  The third author was supported by the NSF via DMS-1301864 and a
Sloan Foundation fellowship}

\date{June 17th, 2015}
\begin{abstract}  Inspired by work of Colding-Minicozzi \cite{CM} on mean
curvature flow, Zhang \cite{Zhang} introduced a notion of entropy stability for
harmonic map flow.  We build further upon this work in several directions. 
First we prove the equivalence of entropy stability with a more computationally
tractable $\FF$-stability.  Then, focusing on the case of spherical targets, we
prove a general instability result for high-entropy solitons.  Finally, we
exploit results of Lin-Wang \cite{LW} to observe long time existence and
convergence
results for maps into certain convex domains and how they relate to generic
singularities of harmonic map flow.
\end{abstract}

\maketitle

\section{Introduction}

In \cite{CM}, Colding and Minicozzi introduced a notion of entropy for mean
curvature flow.  They define a notion of entropy stability for solitons mean
curvature
flow and prove its equivalence to a more computationally tractable notion of
$\FF$-stability.  In this context, a solution to mean curvature flow which only
encounters entropy-stable singularities is said to be generic.  This is in
the sense that one cannot perturb a stable singularity shortly before the
singular time to avoid the singularity. Equipped with the equivalence to
$\calf$-stability, in \cite{CM} it is proved that generic singularities to mean
curvature flow are round spheres or cylinders. 

Inspired by that work, we aim to understand singularities of the harmonic map
heat flow by studying the stability of solitons (see Definition
\ref{solitondef}).  Through rescaling arguments (\cite{Grayson, LW}), one can
produce soliton limits from singularities of the harmonic map flow, and thus
understanding them is central to understanding the long time existence behavior
of the flow.  First we study a entropy stability condition for the
harmonic
map heat flow and prove its equivalence to a notion of $\calf$-stability
introduced in \cite{Zhang}.  The relevant definitions for the statement below
are found in \S \ref{s:euclrestr} and \S \ref{s:entropystability}.
\begin{thm}\label{thm:main1} Suppose $f: \bRm \to N$ is a non-cylindrical
soliton with polynomial energy density growth. Then $f$ is $\calf$-unstable if
and only
if there is a compactly supported variation $f_s$ such that $f_0 = f$ and for
all $s \neq 0$ one has
\begin{equation*}
\la(f_s) < \la(f).
\end{equation*}
\end{thm}

The notion of $\FF$-stability is a more tractable notion involving the spectrum
of a certain kind of second variation operator associated to the the entropy
functional (see (\ref{Lfdef})).  Given this, one would like to simplify the
task of testing for stability further by obtaining a characterization in terms
of Rayleigh quotients.  The quantity $\mu_1$ in the statement below is the
infimum of the relevant Rayleigh quotient (see Definition \ref{mudefn}), which,
due to the noncompactness of the domain, may equal $-\infty$.

\begin{thm} \label{thm:rayleigh} Suppose $f : \bRm \to N$ is a soliton
with polynomial energy density growth.  If $\mu_1 < - \frac{3}{2}$ then $f$ is
$\FF$-unstable. 
\end{thm}

Given the equivalence between $\la$-stability and $\calf$-stability, we then
consider the problem of classifying $\la$-stable solutions into various target
spaces by studying $\calf$-stability via the Rayleigh quotient $\mu_1$.  We
first note that in principle one
should expect a large variety of stable solitons for the harmonic map flow.  In
particular, as shown in \cite{CD}, given $m,n\geq 3$, for any homotopically
non-trivial class in $C^1(S^m,S^n)$ there is an $\epsilon>0$ such that any
$f_0:S^m\to S^n$ in the class with the energy below $\epsilon$ is the initial
condition for a heat flow which goes singular in finite time, leading to soliton
blowup models which should be stable. On the other hand the entropy of these solutions cannot be too small due to the $\ge$-regularity results of Struwe \cite{Struwe1}.  Our result gives a uniform upper bound for the entropy of stable solitons mapping into the sphere, complementing these results.

\begin{thm}\label{thm:main2}  Suppose $f:\bRm\to S^n$ is an entropy stable 
soliton. Then
\begin{equation*}
\gl(f)\leq\tfrac{3n}{4(n-2)}.
\end{equation*}
\end{thm}

To prove Theorem \ref{thm:main2} we first use the equivalence between entropy
stability and $\FF$-stability provided by Theorem \ref{thm:main1}.  We then
construct test variation fields for the $\FF$-stability condition using
conformal vector fields on the sphere.  We cannot directly show the existence of
negative eigenvectors in the space of conformal vector fields and so instead we
rely on Theorem \ref{thm:rayleigh}.  The use of
conformal vector fields in understanding stability of harmonic maps into spheres
is classical, see \cite{RTSMITH}.

Lastly, we observe results relating entropy stability and $\FF$-stability to a
kind of ``dimensional instability'' for harmonic map heat flow.  Taking a cue
from classical results we show in Proposition
\ref{prop:solitonintospheres} that any nonconstant soliton whose image is
contained in a great sphere is entropy unstable.  The vector field exhibiting
instability is the conformal vector field with poles at the vectors orthogonal
to the given great sphere.  This suggests that, in studying harmonic maps into
spheres, at a singular time one could embed the map into a sphere of one higher
dimension as a great sphere, and then perturb the map into the extra dimension
to remove the unstable singularity.

As it turns out this strategy is successful, and any map landing strictly in a
hemisphere admits a smooth long time solution to harmonic map flow.  A
well-known result of Gordon \cite{Gordon} says that any harmonic map into a
region admitting a convex function is constant.  We show a related result
which says that if a map
has image contained in a sublevel set of a boundary-defining convex function,
the solution to
harmonic map
flow with this initial condition exists smoothly for all time and converges to a
point.  Lin-Wang show results of this kind, and indeed our result is
effectively a corollary of (\cite{LW} Theorem 5.4.3).  We require a little
bit of care, as the results of Lin-Wang implicitly assume completeness of the
target, which is not satisfied in our setting.  Indeed, we provide examples
in \S \ref{s:dimaug} that various natural statements about harmonic map flow
with convex
target are false.  These issues are resolved by observing a
parabolic maximum principle which shows that
sublevel sets of boundary-defining convex functions are preserved under harmonic
map flow.

\begin{thm} \label{ltethm} Let $(M^n, g)$ be a compact Riemannian manifold, and
suppose $(N, \del N, h)$ is a compact Riemannian manifold with boundary which is
a sublevel set of a
strictly
convex boundary-defining function.  Given $f : M \to N$ a smooth map, the
solution to harmonic map heat flow with initial condition $f$ exists for all
time and converges to a point.
\end{thm}

This theorem in particular implies the result mentioned above that maps into
hemispheres admit long time solutions to harmonic map flow which converge to
points.  This
suggests a different way to move past singularities for maps into
spheres, by perturbing into a hemisphere of a sphere of one dimension higher. 
Moreover, this result motivates some conjectures about the structure
of solitons mapping into spheres by perturbing the maps into a hemisphere of one
dimension higher and studying the behavior of the flow at infinity.  This is
discussed in \S \ref{ss:conjs}.

Here is an outline of the rest of the paper.  In \S \ref{s:Ffunctentropy} we
record some fundamental properties of the $\FF$-functionals and entropy.  In \S
\ref{s:euclrestr} refine these results in the case of a Euclidean source.  Then
in \S \ref{s:entropystability} we prove Theorem \ref{thm:main1}.  In \S
\ref{s:stabilityrigidity} we give stability and rigidity results for solitons,
focusing on the sphere and leading to Theorem \ref{thm:main2}.  We conclude in
\S \ref{s:dimaug} with the proof of Theorem \ref{ltethm}.

\subsection*{Acknowledgments.}

The second author gratefully acknowledges support from the Simons Center for
Geometry and Physics, Stony Brook University at which some of the research for
this paper was performed. The authors would like to thank Richard Schoen for his
insight and support.

\section{ \texorpdfstring{$\mathcal{F}$}{F}-functional and
entropy}\label{s:Ffunctentropy}

In this section we provide some background and set notation concerning harmonic
maps, and provide the basic definitions of the $\mathcal{F}$-functional and its
stability.

\subsection{Background and notation}
Let $(M^m,g)$ and $(N^n,h)$ be two closed Riemannian manifolds, and $f \in
C^{\infty}(M,N)$. We will consider a connection $\N$ defined on the tensor
bundle $TM \ten f^* (TN)$ induced by the Levi-Civita connection on $M$ and $N$.
We will drop all superscripts written on the connections unless necessary, and
will generally abuse notation by letting $\N$ denote the various connections on
different bundles.

When considering local coordinates about some point, we will use Greek indices
$\{ x^{\ga} \}_{\ga = 1}^m$ for coordinates on $M$ and Roman indices $\{ y^{i}
\}_{i=1}^n$ for coordinates on $N$. Furthermore, we will write $\del_{\ga} :=
\frac{\del}{\del x^{\ga}}$ and $\del_{i} := \frac{\del}{\del y^{i}}$ to span
$TN$. We abuse notation by letting $\del_i$ simultaneously denote an element of
$TN$ and $f^*(TN)$. The associated duals will be given by $d^{\ga}$ and $d^i$,
respectively.
Let $Tf:TM\to f^*(TN)$ be the \emph{differential of $f$}, given in local
coordinates by
\begin{align*}
Tf=\prs{\del_\alpha f^i} d^\ga\otimes\del_i.
\end{align*}
More generally, subscripts will indicate spatial derivatives of various items.
The action of $\N$ on elements of $f^*(TN)$ has the coordinate form
\begin{equation*}
\N_{\ga} \del_j = \prs{\del_\alpha f^i} \N_i \del_j = \prs{\del_\alpha
f^i}\gG^k_{ij}\del_k,
\end{equation*}
where $\gG_{ij}^k$ denotes the connection coefficients of the Levi-Civita
connection in local coordinates on $N$.

\subsection{Entropy and solitons}

For a smooth map $f \in C^{\infty} ((M,g), (N,h))$, we define the \emph{energy
density}
by
\begin{align*}
e(f) := \frac{1}{2} \brs{Tf}^2_{g,h},
\end{align*}
and the \emph{total energy} by
\begin{align*}
\mathcal E(f) := \int_M e(f) dV_g.
\end{align*}
The \emph{tension field}, which is the negative gradient of
the total energy, is given by
\begin{equation*}
\tau(f) := \tr_g (\N Tf),
\end{equation*}
The \emph{harmonic map heat flow}, or negative gradient flow of the
energy, is given by
\begin{equation}\label{eq:HMHF}
\frac{\del f_t}{\del t} = \tau(f_t). \tag{\textsf{HMHF}}
\end{equation}
We will suppress the dependence on $f_t$ in the notation for $\tau_t$ from this
point forward. We next define the $\FF$-functional as well as the entropy for a
map
of Riemannian manifolds.

\begin{defn} \label{Ffunc} Given $(M^m, g)$ and $(N^n, h)$ Riemannian manifolds,
define a functional $\FF$ via
\begin{align}\label{eq:Ffunc}
\begin{split}
\mathcal{F} &:C^{\infty}(M,N) \times \mathbb{R} \times  C^{\infty}(M)  \to
\mathbb{R} \\
\FF& (f, \upsilon, \theta) = \frac{\upsilon}{2} \int_M \brs{Tf}^2_{g,h}
\frac{e^{-\theta}}{(4 \pi \upsilon )^{m/2}} dV_g.
\end{split}
\end{align}
Moreover we define the \emph{entropy} $\gl$ by
\begin{align}\label{eq:defEntropy}
\begin{split}
\lambda
&: C^{\infty}(M,N) \to \mathbb{R} \\
\lambda & (f) = \sup_{\left\{ \upsilon, \theta : \int_M  e^{-\theta}(4 \pi
\upsilon )^{-m/2}  dV_g = 1 \right\} } \mathcal{F}(f,\upsilon,\theta).
\end{split}
\end{align}
\end{defn}

Zhang \cite{Zhang} defines a functional of this kind for maps from $\mathbb
R^m$, in direct analogy with Colding-Minicozzi's definition for mean curvature
flow.  There one only considers weighting against Gaussian densities with
different basepoints in $\mathbb R^m \times \mathbb R_{\geq 0}$.  This suffices
for the purposes of understanding singularity models in both cases.  We have
generalized
this to a more flexible entropy functional more akin to Perelman's Ricci flow
entropy
\cite{Perelman}, now defined on arbitrary Riemannian manifold by allowing
weighting against arbitrary probability measures.

We will adapt
Hamilton-Struwe's monotonicity formula \cite{Ham2, Struwe1} to establish a
monotonicity for
$\lambda$ in \S \ref{ss:monotonicity}.  These monotone quantities $\FF$ and
$\lambda$ are central to understanding singularities of harmonic map flow. 
Moreover, the critical points of $\lambda$ are exactly self shrinking solutions.
The primary disadvantage of $\lambda$ is that it does not depend smoothly on
$C^{\infty}(M,N)$. Therefore, to overcome this in using $\lambda$ to define the
stability of solitons, we will 
demonstrate that it is essentially sufficient to check classical stability of
the functional $\mathcal{F}$.  This analysis will be performed in \S
\ref{s:entropystability} in the case where the source manifold is flat Euclidean
space.

\subsection{Variational properties}\label{ss:variational}

Next we compute the first and second variations of the
$\mathcal{F}$-functional (Definition \ref{Ffunc}). In the next section \S
\ref{s:euclrestr} we will restrict to the case of Euclidean space, and consider
a more restricted form of the entropy functional $\lambda$.  Throughout this
section we will suppose that the given geometric data is regular enough so that
all integrations by parts are valid.  To begin we define a family of flows which
will be key in the variations.

First, given a solution to harmonic map heat flow on $[0,t_0)$, and a final
value $\Psi_{t_0} : M \to \mathbb{R}$ we consider the \emph{backwards heat flow}
given by
\begin{equation}\label{eq:BHF}
\begin{cases} \tag{\textsf{BHF}}
\tfrac{\del \Psi_t}{\del t} &= - \lap \Psi_t, \\
\left. \Psi_t \right|_{t=t_0} & = \Psi_{t_0}.
\end{cases}
\end{equation}
In the case that $\Psi_{t_0}$ is a Dirac delta function based at $x_0$, we say
that this is the \emph{backwards heat flow based at $(x_0,t_0)$}.
Next we consider the following two flows: first,
\begin{equation}\label{eq:LF}
\begin{cases}
\del_t \upsilon_t &= -1, \\
\left. \upsilon_t \right|_{t=0} &= \upsilon_0,
\end{cases}
\end{equation}
and then
\begin{equation}\label{eq:TF}
\begin{cases}
\del_t \theta_t &= \prs{ \brs{ \N \theta_t}^2 - \lap \theta_t } +  \tfrac{m}{
2\upsilon_t}, \\
\left. \theta_t \right|_{t = 0} &= \theta_0.
\end{cases}
\end{equation}

\begin{lemma}\label{lem:Thetahteq}
Suppose $\upsilon_t$ satisfies \eqref{eq:LF} and $\theta_t$ satisfies
\eqref{eq:TF}. Then
$\Theta_t := (4 \pi \upsilon_t)^{-m/2}e^{-\theta_t}$ satisfies \eqref{eq:BHF}.

\begin{proof}
We simply compute
\begin{align*}
\lap \Theta_t
&= \N_{\ga} \N_{\ga} \brk{ \tfrac{e^{- \theta_t}}{(4 \pi \upsilon_t)^{m/2}} } \\
&= \N_{\ga}\brk{  \tfrac{- \prs{\N_{\ga} \theta_t} e^{- \theta_t}}{(4 \pi
\upsilon_t)^{m/2}} } \\
&=  \tfrac{- ( \lap \theta_t)}{(4 \pi \upsilon_t)^{m/2}} + \tfrac{\brs{ \N
\theta}^2 e^{- \theta_t}}{(4 \pi \upsilon_t)^{m/2}} \\
&= \tfrac{e^{-\theta_t}}{(4 \pi \upsilon_t)^{m/2}} \prs{ \brs{\N \theta_t}^2 -
\lap \theta_t } \\
& =  \tfrac{e^{-\theta_t}}{(4 \pi \upsilon_t)^{m/2}} \prs{ (\del_t \theta_t) -
\tfrac{m}{2 \upsilon_t} } \\
&= - \prs{ \tfrac{4 \pi m}{2} \tfrac{e^{-\theta_t}}{(4 \pi \upsilon_t)^{m/2 +
1}} - \tfrac{(\del_t \theta_t)  e^{-\theta_t}}{(4 \pi \upsilon_t)^{m/2}} } \\
& = - \del_t \Theta_t.
\end{align*}
The result follows.
\end{proof}
\end{lemma}

\begin{defn}
Given $\Theta \in C^{\infty}(M)$, $\Theta > 0$, let 
\begin{align*}
\textsf{S}_{\Theta}
&: C^{\infty}(M,N) \to  f^*(TN) \\
\textsf{S}_{\Theta}& (f) = \tau +  \prs{\frac{\N \Theta}{\Theta} \hook Tf }.
\end{align*}
\end{defn}

\begin{prop}\label{prop:1stvarM} Given one-parameter families
$f_s$, $v_s$, $\theta_s$, let
\begin{equation*}
\dot{\upsilon}_s = \frac{d \upsilon_s}{ds}, \qquad \dot{f}_s = \frac{d f_s}{ds},
\qquad \dot{\theta}_s = \frac{d \theta_s}{ds}.
\end{equation*}
Moreover, let $\Theta_s = (4 \pi \upsilon_s)^{-m/2} e^{-\theta_s}$ as above. 
Assuming that $f_s \in H_{loc}^1(M, N)$ and ${f}_s$ satisfies
\begin{equation*}
\int_M \prs{ \brs{\dot{f}_s}^2 + \brs{\N f_s}^2 + \brs{ \tfrac{\N
\Theta_s}{\Theta_s} }^2 \brs{ T f_s}^2 + \brs{\tau_s}^2} \Theta_s dV_g < +
\infty,
\end{equation*}
then
\begin{align}
\begin{split}\label{eq:1stvarM}
\frac{d}{ds} \lb \mathcal{F} ( f_s,\upsilon_s, \theta_s) \rb &=  \frac{1}{2}
\int_{M} \left(\dot{\upsilon}_s  + \upsilon_s \tfrac{\dot{\Theta}_s}{\Theta_s}
\right) | T f_s |^2  \Theta_s  dV_g  - \upsilon_s \int_{M} \left\langle
\dot{f}_s, \mathsf{S}_{\Theta_s}(f_s) \right\rangle \Theta_s dV_g.
\end{split}
\end{align}

\begin{proof}
We first compute, and apply integration by parts to obtain
\begin{align}
\begin{split}\label{eq:1stvarM10}
\frac{d}{ds} \lb \int_{M} | Tf |^2 \Theta dV_g \rb
&= 2 \int_{M}  \langle \N_s T f_{\ga} , Tf_{\ga} \rangle \Theta dV_g + \int_{M} 
| Tf |^2 \dot{\Theta} dV_g\\
&= 2 \int_{M} \langle \N \dot{f}, Tf \rangle \Theta dV_g + \int_M  | Tf |^2
\dot{\Theta} dV_g.
\end{split}
\end{align}
We manipulate the first integral,
\begin{align*}
\begin{split}
  \int_{M} \langle \N \dot{f}, Tf \rangle \Theta dV_g
  &=   \int_{M} \del_{\ga} \langle  \dot{f}, T f_{\ga} \rangle \Theta dV -
\int_{M} \langle  \dot{f}, \N_{\ga} T f_{\ga} \rangle \Theta dV_g\\
  &=  - \int_{M} \langle  \dot{f}, T f_{\ga} \rangle  \prs{\del_{\ga} \Theta}
dV_g - \int_{M} \langle  \dot{f}, \N_{\ga} T f_{\ga} \rangle \Theta dV_g\\
&= - \int_{M}  \left\langle \dot{f}, \prs{\frac{\N \Theta}{\Theta} \hook Tf }
\right\rangle  \Theta dV_g-  \int_{M} \langle \dot{f}, \tau \rangle \Theta dV_g
\\
&= - \int_{M}  \left\langle \dot{f}, \mathsf{S}_{\Theta}(f) \right\rangle 
\Theta dV_g.
\end{split}
\end{align*}
Incorporating this back into \eqref{eq:1stvarM10} we have that
\begin{align*}
\frac{d}{ds} \lb \int_{M} | Tf_s |^2  \Theta_s dV_g \rb
&= \int_{M}\tfrac{\dot{ \Theta_s}}{ \Theta_s} | Tf_s |^2    \Theta_s  dV_g  - 2
\int_{M} \left\langle \dot{f}_s, \mathsf{S}_{ \Theta_s}(f_s) \right\rangle  
\Theta_s dV_g.
\end{align*}
Therefore
\begin{align*}
\frac{d}{ds} \lb \mathcal{F} ( f_s,\upsilon_s, \theta_s) \rb 
&= \frac{\dot{\upsilon}_s}{2} \int_{M} |T f_s|^2  \Theta_s dV +
\frac{\upsilon_s}{2} \frac{d}{ds} \lb \int_{M} |T f_s|^2  \Theta_s dV_g \rb \\
&=   \frac{1}{2} \int_{M} \left(\dot{\upsilon}_s  + \upsilon_s \tfrac{\dot{
\Theta_s}}{ \Theta_s} \right) | T f_s |^2   \Theta_s  dV_g  - \upsilon_s
\int_{M} \left\langle \dot{f}_s, \mathsf{S}_{ \Theta_s} (f_s) \right\rangle 
\Theta_s dV_g.
\end{align*}
The result follows.
\end{proof}
\end{prop}

Before computing the second variation we record a lemma showing how $\tau$
varies along a path.  The cleanest way to do this is to treat the family as a
map on $M \times I$, and then one naturally observes curvature terms arising
when commuting time derivatives with
connection derivatives. An example of this is in the following lemma.

\begin{lemma}\label{lem:vartau} Suppose $f_s$ is a one-parameter family of
smooth maps for $s \in I \subset \mathbb{R}$. Then
\begin{equation*}
\left( \frac{\del \tau_s}{\del_s} \right) =  \lap \dot{f}_s + R^N(\dot{f}_s,
(Tf_s)_{\ga})(Tf_s)_{\ga}.
\end{equation*}

\begin{proof}
First one observes the commutation formula
\begin{equation}\label{eq:comutform}
\brk{\N_{\ga},\N_{\gb}} \gw_{\gz}^i
= \prs{R^M}_{\gb \ga \gz}^{\gd} \gw_{\gd}^i +  (\del_{\gb} f^l)(\del_{\ga} f^v)
( R^N)_{v l k}^i \gw_{\gz}^k.
\end{equation}
Using this we have 
\begin{align*}
\left( \frac{\del \tau}{\del_s} \right)^{i}
&= \N_s \N_{\ga} \prs{\del_{\ga} f^i} \\
&=  [\N_s, \N_{\ga}] \prs{\del_{\ga} f^i} + \N_{\ga} \prs{\del_{\ga}\dot{f}^i} 
\\
&= (\del_{\ga} f^{p})(\del_{\ga} f^j) (\dot{f}^q) (R^N)_{pqj}^{i} - (R^M)_{s \ga
\ga}^{\gd} (\del_{\ga} f^i) + \lap \dot{f}^i \\
&= (\del_{\ga} f^{p})(\del_{\ga} f^j) (\dot{f}^q) (R^N)_{pqj}^{i} + \lap
\dot{f}^i.
\end{align*}
The result follows.
\end{proof}
\end{lemma}

\begin{prop} \label{prop:2ndvarM}
Assuming the hypotheses of Proposition \ref{prop:1stvarM}, one has
\begin{align}
\begin{split}\label{eq:2ndvarM}
\tfrac{d^2}{ds^2}\brk{\mathcal{F}(f_s,\nu_s,\theta_s)}
&=\frac{\ddot{\upsilon}_s}{2}\int_M|Tf_s|^2\Theta_s dV_g- \upsilon_s\int_M \ip{
\ddot{f}_s,\tau_s +\prs{\tfrac{\N\Theta_s}{\Theta_s} \hook
Tf_s}}\Theta_sdV_g+\frac{\upsilon_s}{2}\int_M |Tf_s|^2\ddot{\Theta}_s dV_g\\
&\hsp -\dot{\upsilon}_s\int_M \prs{\ip{
\dot{f}_s,\tau_s+\prs{\tfrac{\N\Theta_s}{\Theta_s} \hook Tf_s}}
\Theta_s-|Tf_s|^2 \dot{\Theta}_s }dV_g\\
& \hsp -\upsilon_s\int_M \ip{ \dot{f}_s,(\lap
\dot{f}_s+R^N(\dot{f},(Tf_s)_{\ga})(Tf_s)_{\ga}) +\N_{\tfrac{\N
\Theta_s}{\Theta_s}} \dot{f}_s} \Theta_s dV_g \\
& \hsp - 2 \upsilon_s \int_M \ip{ \dot{f}_s, \tau_s
\prs{\tfrac{\dot{\Theta}_s}{\Theta_s}} +\prs{\tfrac{\N \dot{\Theta}_s}{\Theta_s}
\hook Tf_s} } \Theta_s dV_g.
\end{split}
\end{align}
\begin{proof}
We rewrite \eqref{eq:1stvarM} in the form
\begin{align*}
\tfrac{d}{ds}\brk{\calf(f_s,\upsilon_s,\theta_s)}
&=\brk{\frac{\dot{\upsilon}_s}{2}\int_M|Tf_s|^2\Theta_s dV_g}_{T_1} - \brk{
\upsilon_s \int_M \ip{ \dot{f}_s,\tau_s +\prs{\frac{\N\Theta_s}{\Theta_s} \hook
Tf_s } }\Theta_s dV_g}_{T_2} \\
& \hsp +\brk{\frac{\upsilon_s}{2}\int_M \brs{Tf_s}^2 \dot{\Theta}_s dV_g}_{T_3}
\\
&= T_1 - T_2 + T_3.
\end{align*}
We differentiate the term $T_1$ to yield
\begin{align*}
\frac{d}{ds}\brk{T_1}
&= \frac{\ddot{\upsilon}}{2}\int_M |Tf|^2\Theta dV_g+\dot{\upsilon} \int_M
\ip{\N \dot{f},Tf }\Theta dV_g+\frac{\dot{\upsilon}}{2}\int_M|Tf|^2 \dot{\Theta}
dV_g \\
&= \frac{\ddot{\upsilon}}{2}\int_M |Tf|^2\Theta dV_g-\dot{\upsilon} \int_M
\ip{\dot{f}, \tau +\prs{\tfrac{\N \Theta}{\Theta} \hook  Tf } }\Theta
dV_g+\frac{\dot{\upsilon}}{2}\int_M|Tf|^2 \dot{\Theta} dV_g.
\end{align*}
Using Lemma \ref{lem:vartau}, the variation of the second term gives
\begin{align*}
\frac{d}{ds}\brk{T_2}
&=\dot{\upsilon}\int_M \ip{ \dot{f},\tau +\prs{\tfrac{\N \Theta}{\Theta} \hook 
Tf }} \Theta dV_g + \upsilon \int_M\ip{ \ddot{f},\tau+\prs{\tfrac{\N
\Theta}{\Theta} \hook  Tf }} \Theta dV_g\\
&+\upsilon \int_M \ip{ \dot{f},(\lap \dot{f}+R^N(\dot{f},Tf_{\ga})Tf_{\ga})+\tau
\prs{\tfrac{\dot{\Theta}}{\Theta}}+\N \dot{f}
\prs{\tfrac{\N\Theta}{\Theta}}+\prs{\tfrac{\N \dot{\Theta}}{\Theta} \hook  Tf }
} \Theta dV_g.
\end{align*}
Finally the variation of the third term is, with an application of
integration by parts to the middle quantity,
\begin{align*}
\tfrac{d}{ds} \brk{T_3} 
&= \frac{\dot{\upsilon}}{2}\int_M \brs{Tf}^2 \dot{\Theta}dV_g + \upsilon \int_M
\ip{\N \dot{f},Tf }\dot{\Theta} dV_g +\frac{\upsilon}{2}\int_M|Tf|^2
\ddot{\Theta}
dV_g\\
&= \frac{\dot{\upsilon}}{2}\int_M \brs{Tf}^2 \dot{\Theta}dV_g  -\upsilon \int_M
\ip{ \dot{f},\tau \prs{\tfrac{\dot{\Theta}}{\Theta}} + \prs{ \tfrac{\N
\dot{\Theta}}{\Theta} \hook Tf}} \Theta dV_g+\frac{\upsilon}{2}\int_M|Tf|^2
\ddot{\Theta} dV_g.
\end{align*}
Combining and regrouping one obtains \eqref{eq:2ndvarM}.
\end{proof}
\end{prop}

\subsection{Monotonicity}\label{ss:monotonicity}

The following monotonicity formulas from \cite{Struwe1} and \cite{Zhang} are
essential for the computations of the proof of Theorem \ref{thm:main1} in \S
\ref{s:entropystability}. We state them, then discuss generalized versions.

\begin{lemma}[\cite{Struwe1} Lemma
3.2]\label{thm:Struwemonotonicity} Let $f_t \in C^{\infty} (\bRm
\times [0,T), N)$ be a solution to \eqref{eq:HMHF} with uniformly bounded energy
density.  Then
\begin{equation*}
\frac{d}{dt} \brk{ \mathcal{F} \prs{ f_t, t_0 - t, \tfrac{\brs{x-x_0}^2}{4(t_0 -
t)}}} = -(t_0 -t) \int_{\bRm} \brs{ \tau_t - \prs{ \frac{x-x_0}{2 (t_0 - t)}
\hook Tf_t }}^2 \frac{e^{ - \frac{ \brs{x-x_0}^2}{4 (t_0 - t)}}}{(4 \pi (t_0 -
t))^{m/2}}  dV.
\end{equation*}
\end{lemma}
\noindent Using this and the first variation (cf. Corollary \ref{prop:1stvar}),
Zhang
proves the following entropy monotonicity.
\begin{prop}[\cite{Zhang} Proposition 4]\label{prop:Ftmonotonicity} Let $f_t \in
C^{\infty}(\mathbb{R}^m \times [0,T), N)$ be a solution to \eqref{eq:HMHF} with
uniformly bounded energy density. Then the quantity
\begin{equation*}
\la(f_t)  := \sup_{t \in \mathbb{R}, x \in \mathbb{R}^m} \mathcal{F}\prs{
f_t,t_0 - t, \tfrac{ \brs{x-x_0}^2}{ 4 (t_0 -t)} },
\end{equation*}
is non increasing in $t$.
\end{prop}

One may utilize Proposition \ref{prop:1stvarM} to give a generalization of Lemma
\ref{thm:Struwemonotonicity}, in the case of general source manifolds $M$ rather
than $\mathbb{R}^m$. This is proven by Hamilton in \cite{Ham2}.
\begin{cor}[\cite{Ham2} Theorem 1.1, pp.131]\label{lem:BHFFmon}
 Let $f_t \in C^{\infty} (\bRm \times [0,T), N)$ be a solution to
\eqref{eq:HMHF} with uniformly bounded energy, and let $\upsilon_t$ and
$\theta_t$ satisfy \eqref{eq:LF} and \eqref{eq:TF} respectively. Set $\Theta_t
:= e^{- \theta_t}(4 \pi \upsilon_t)^{-m/2}$.  Then
\begin{align}\label{eq:BHFFmon}
\begin{split}
\tfrac{d}{dt} \brk{ \mathcal{F}(f_t,\upsilon_t, \theta_t) } 
=&\ - \upsilon_t \int_{M} \brs{ \tfrac{\prs{\N \Theta_t} \hook Tf_t }{\Theta_t}
+ \tau_t }^2 \Theta_t  dV_g  \\
&\ + \upsilon_t \int_M{ g^{\ga \gb} g^{\gz \gd} \ip{ \prs{Tf_t}_{\ga},
\prs{Tf_t}_{\gd}} \prs{ \tfrac{\prs{\N_{\gz} \Theta_t} \prs{\N_{\gb} \Theta_t} 
- \N_{\gz} \N_{\gb} \Theta_t }{\Theta_t} - \prs{\tfrac{g_{\gz\gb} }{2
\upsilon_t}}} } \Theta_t  dV_g. 
\end{split}
\end{align}
\end{cor}

We next state the resultant monotonicity formula, given in Hamilton's paper
\cite{Ham2}, which is a result of applying the Harnack estimates from
\cite{Ham1} to control the last term of \eqref{eq:BHFFmon}.

\begin{thm}[\cite{Ham2} Theorem A, pp.123] Let $f_t \in C^{\infty}(M \times
[0,T),N)$ solve
\eqref{eq:HMHF} on $0 \leq t < T$ and $\upsilon_t$ and $\theta_t$ satisfy
\eqref{eq:LF} and \eqref{eq:TF} respectively. Set $\Theta_t := e^{- \theta_t}(4
\pi \upsilon_t)^{-m/2}$. Then if $\int_M \Theta_t = 1$, the quantity
$\mathcal{F}(f_t,\upsilon_t, \theta_t)$ is monotone decreasing in $t$ when $M$
is Ricci parallel with weakly positive sectional curvatures. On a general
manifold $M$ we have
\begin{equation*}
\mathcal{F}(f_t,\upsilon_t, \Theta_t) \leq C_M
\prs{\mathcal{F}(f_{\tau},\upsilon_{\tau}, \Theta_{\tau})  + (t- \tau)
\int_M\brs{T f_0}^2 dV_g},
\end{equation*} 
for $T-1 \leq \tau \leq t \leq T$, and $C_M \in \mathbb{R}_{>0}$ a constant
depending only on $M$.
\end{thm}

\begin{cor}\label{cor:entropymon} Let $f_t \in
C^{\infty}(M \times [0,T),N)$ be a solution to \eqref{eq:HMHF} with uniformly
bounded energy for $t \in [0,T)$. Then we have that
\begin{equation}\label{eq:entropymon}
\lambda (f_t) \leq C_M \prs{ \lambda( f_{\tau}) + (t-\tau) \int_M\brs{T f_0}^2
dV_g }.
\end{equation}
for $T-1 \leq \tau \leq t \leq T$, and $C_M \in \mathbb{R}_{>0}$ a constant
depending only on $M$. Moreover, if $M$ is Ricci parallel with weakly
positive sectional curvatures, then
$\lambda (f_t)$ is monotone. 

\begin{proof}
First assume $M$ is Ricci parallel with weakly positive sectional curvatures.
Suppose to the contrary that
there exists some solution to harmonic map heat flow, $f_t$ such that
$\lambda(f_{t_0}) < \lambda(f_{t_1})$ for $t_0 < t_1$. Then we have that there
exists some choice of $\upsilon_{t_1}$ and $\theta_{t_1}$ such that
\begin{equation*}
\lambda(f_{t_0}) < \mathcal{F}( f_{t_1}, \upsilon_{t_1}, \theta_{t_1}).
\end{equation*}
If we let $\upsilon_t, \theta_t$ as above satisfying \eqref{eq:LF} and
\eqref{eq:TF} with the appropriate final values, then by Corollary
\ref{lem:BHFFmon} we have that $\tfrac{d}{dt} \brk{ \mathcal{F}(
f_t,\upsilon_t, \theta_t )} \leq 0$, so we have that 
\begin{align*}
\lambda(f_{t_0}) \geq \mathcal{F}( f_{t_0}, \upsilon_{t_0}, \theta_{t_0}) \geq
\mathcal{F}( f_{t_1}, \upsilon_{t_1}, \theta_{t_1}) > \lambda(f_{t_0}).
\end{align*}
This is clearly a contradiction; the first result follows.

More generally assume that equality \eqref{eq:entropymon} is violated for some
times $t_0 < t_1$. Then we have that
\begin{equation*}
\lambda(f_{t_1}) > C_M \prs{ \lambda(f_{t_0}) + (t_1 - t_0) \int_M \brs{Tf_0}^2
dV_g}.
\end{equation*}
Then one can choose $\theta_t$, $\upsilon_t$ such that
\begin{equation*}
\mathcal{F}(f_{t_1}, \upsilon_{t_1}, \theta_{t_1}) > C_M \prs{ \lambda(f_{t_0})
+ (t_1 - t_0) \int_M \brs{Tf_0}^2 dV_g},
\end{equation*}
and moreover using Corollary \ref{lem:BHFFmon} we have that
\begin{align*}
\lambda(f_{t_0}) 
& \geq \tfrac{1}{C_M} \mathcal{F}(f_{t_1},\upsilon_{t_1},\theta_{t_1}) - (t_1 -
t_0) \int_M \brs{T f_0}^2 dV_g \\
& > \tfrac{1}{C_M} \prs{C_M  \lambda(f_{t_0},\upsilon_{t_0},\theta_{t_0} + C_M
(t_1 - t_0) \int_M \brs{Tf}^2 dV_g}  - (t_1 - t_0) \int_M \brs{T f_0}^2 dV_g \\
& = \lambda(f_{t_0},\upsilon_{t_0},\theta_{t_0}).
\end{align*}
This is a contradiction, so \eqref{eq:entropymon} follows.
\end{proof}
\end{cor}

\section{Entropy for Euclidean source}\label{s:euclrestr}

In the previous section we established a general notion of entropy for maps
between Riemannian manifolds, and its relationship to harmonic map flow.  As
exhibited in \cite{Grayson}, type I singularities to harmonic map flow admit
blowup
solutions which are shrinking solitons, mapping from flat Euclidean space to the
given target manifold.  The rest of this paper focuses primarily on the study of
shrinking solitons, and so we now restrict to the setting where
$(M,g)=(\mathbb{R}^m, g_{\Euc})$, where $g_{\Euc}$ denotes the standard flat
metric.  Also
henceforth $R$ will represent the curvature tensor on $N$.

In this section we review some fundamental properties of solitons and their
relationship to the $\FF$-functionals.  Many of these results appeared in
\cite{Zhang}.  We are forced to revisit some of these calculations because the
proof of Theorem \ref{thm:main1} requires us to consider variations of the
$\FF$ functionals for general basepoints.

\subsection{\texorpdfstring{$\FF$}{F}-functional revisited}

To begin we define recall a simpler version of the $\mathcal{F}$-functional with
designated fixed base point, defined in \cite{Zhang} in analogy with \cite{CM}. 
For $x_0 \in \bRm$ and $t_0 >0$, define the $\mathcal{F}_{x_0,t_0}$-functional
by
\begin{align*}
\mathcal{F}_{x_0,t_0} 
&: C^{\infty}(\mathbb{R}^m,N)  \to \mathbb{R},\\
\FF_{x_0,t_0}&(f) = \frac{t_0}{2} \int_{M} | T f |^2 \frac{e^{-
\frac{\brs{x-x_0}^2}{4 t_0}}}{(4 \pi t_0)^{m/2}} dV.
\end{align*}
To connect to the notation of \S \ref{s:Ffunctentropy}, observe that
$\mathcal{F}_{x_0,t_0}(f) := \mathcal{F} \prs{ f, t_0, \tfrac{
\brs{x-x_0}^2}{(4 t_0)^{m/2}} }$.  Furthermore we define the \emph{entropy
functional}
by
\begin{equation*}
\la(f) = \sup_{x_0 \in \mathbb R^m, t_0 >0 } \mathcal{F}_{x_0,t_0}(f).
\end{equation*}
As established in \cite{Zhang}, critical points of the $\FF_{x_0,t_0}$-
functional are self-similar shrinking solutions to harmonic map flow. 
\begin{defn} \label{solitondef} A solution $f_t : \mathbb{R}^m \to N$ to
\eqref{eq:HMHF} is called a \emph{$(x_0,t_0)$-self similar solution} if it
satisfies
\begin{equation*}
\tau_t =\prs{ \frac{x-x_0}{2 (t_0 - t)}} \hook Tf_t .
\end{equation*}
Due to the self similarity, it is enough to consider the $t = 0$ time slice of a
self similar solution, which captures all the information of the flow itself
(cf. Proposition \ref{prop:solitonchars}). We call the $t=0$ slice a
\emph{$(x_0,t_0)$-soliton}, given by
\begin{equation*}
\tau_0 = \prs{ \frac{x-x_0}{2 t_0}} \hook T f_0 .
\end{equation*}
The subscript will be omitted when understood. Lastly, taking
$(x_0,t_0) = (0,1)$ we define simply a \emph{soliton}, given by
\begin{equation*}
\tau = \prs{ \frac{x}{2} \hook T f} .
\end{equation*}
\end{defn}

We let $\mathfrak{S}_{x_0,t_0}$ denote the space of all $(x_0,t_0)$-solitons.
We define the \emph{$(x_0,t_0)$-soliton operator} by
\begin{align*}
\mathsf{S}_{x_0,t_0}
&: C^{\infty}(\mathbb{R}^m,N) \to T^*M \ten f^*(TN) \\
\mathsf{S}_{x_0,t_0} &(f) = \tau - \prs{ \frac{x-x_0}{2 t_0} \hook T f }.
\end{align*}
Note in particular $\mathfrak{S}_{x_0,t_0} = \ker \mathsf{S}_{x_0,t_0}$, and
furthermore that $\mathsf{S}_{x_0,t_0}$ is equivalent to $\mathsf{S}_{\Theta}$
for $\Theta = \exp \prs{\tfrac{\brs{x-x_0}^2}{4 t_0}}$. Furthermore, the set
$\mathfrak{S}_{x_0,t_0}$ is equivalent to $\mathfrak{S} := \mathfrak{S}_{0,1}$
via appropriate coordinate shifts, as demonstrated in the following lemma.
\begin{lemma} The sets $\mathfrak{S}_{x_0,t_0}$ and $\mathfrak{S}$ are in
bijective correspondence for all $x_0 \in \mathbb{R}^m$ and $t_0 >0$.
\begin{proof}
Suppose that $f \in \mathfrak{S}$. Then set
\begin{equation*}
\widetilde{f}(x) := f \prs{ \tfrac{x-x_0}{t_0^{1/4}}}.
\end{equation*}
Note that
\begin{align*}
\tau_{\widetilde{f}} (x) 
&= \N_{\ga} \del_{\ga} \brk{ f^i \prs{\tfrac{x-x_0}{t_0^{1/4}}}} \\
&= \tfrac{1}{\sqrt{t_0}} \tau_{f} \prs{ \tfrac{x-x_0}{\sqrt{t_0}}} \\
&=  \tfrac{1}{\sqrt{t_0}} Tf_{\ga} \prs{ \tfrac{(x-x_0)}{2\sqrt{t_0}} ^{\ga}} \\
&= Tf_{\ga} \prs{ \tfrac{x-x_0}{2 t_0}}^{\ga}.
\end{align*}
Thus we have defined a mapping to $\mathfrak{S}_{x_0,t_0}$.  The inverse map
sends $f \in \mathfrak{S}_{x_0,t_0}$ to $\bar{f}(x) := f(x_0 + \sqrt{t_0} x)$,
as is easily checked.
\end{proof}
\end{lemma}

There are a multitude of ways to characterize solitons; we state and explicitly
demonstrate some in the following proposition. We prove the proposition in a
series of lemmas, marked following each corresponding item.  Item (3) is
postponed to later in the section.

\begin{prop} \label{prop:solitonchars} A
smooth map $f : \mathbb{R}^m \to N$ is a soliton if and only if any
of the following holds.
\begin{enumerate}
\item (Lemma \ref{lem:char1}). The one-parameter family $f_t := \sqrt{-t}
f \subset C^{\infty}(\mathbb{R}^m, N)$ satisfies \eqref{eq:HMHF}. 
\item (Lemma \ref{lem:char2}) $f$ is harmonic with respect to the
conformal metric $g := e^{-
\tfrac{\brs{x}^2}{4(m-2)}} g_{\Euc}$.
\item (Corollary \ref{cor:critpoints}) $f$ is a critical point for the
$\mathcal{F}_{0,1}$-functional.
\end{enumerate}
\end{prop}

\begin{lemma} \label{lem:char1} A smooth map $f : \mathbb{R}^m \to N$
is a soliton if and only if the one-parameter family
\begin{equation} \label{lem:char110}
f_t(x) = f_{ -1}\left( \tfrac{x}{\sqrt{-t}} \right), \qquad t <0,
\end{equation}
is a solution to \eqref{eq:HMHF}.
\begin{proof}
First assume that (\ref{lem:char110}) holds.  Take, for $\kappa \in
\mathbb{R}_{>0}$,
\begin{equation*}
f(x,t) := f\prs{ \kappa x, \kappa^2 t}.
\end{equation*}
Differentiating with respect to $\kappa$ and evaluating at $\kappa = 1$ yields
\begin{align*}
0
&=\tfrac{\del}{\del \kappa} \prs{ \left. f_{\kappa} - f }\right|_{\kappa = 1} \\
&=\left. x^{\ga} \del_{\ga} f_{\kappa} + 2t \kappa \del_t f_{\kappa}
\right|_{\kappa = 1} \\
&=  x \hook Tf  + 2 t \del_t f.
\end{align*}
Therefore since $\del_t f = \tau$ we have the desired soliton equation.
Consequently if we take $\kappa := (-t)^{-1/2}$ we have the desired result.

Next assume $\bar{f}_t$ is another solution to the self similar solution
equation with $\bar{f}_{-1} \equiv f$. Take $E : N \to \mathbb{R}^n$ to be some
smooth embedding.  Define $\psi_t := E(f_t) - E(\bar{f}_t)$. Then by the chain
rule, for each $t$ we have
that $\psi_t$ is in the kernel of the operator
\begin{equation*}
\Psi: C^{\infty}(\mathbb{R}^m \times \mathbb{R}_{\leq 0},\bRn) \to T\bRn :
\psi \mapsto \tfrac{\del \psi}{\del t} + T \psi_t \prs{ \tfrac{x}{2t}}.
\end{equation*}
We verify that for any $s \in \mathbb{R}_{\leq 0}$ the hypersurface
$\mathbb{R}^m \times
\{s\}$ is non-characteristic with respect to the operator $\Psi$, which is
equivalent to demonstrating that $\Psi$ is non degenerate in the transverse
direction of the boundary of $\mathbb{R}^m \times \{ s \}$, namely,
\begin{equation*}
\left\langle \sigma[\Psi], \del_t \right\rangle \neq 0,
\end{equation*}
where here $\sigma[\Psi]$ denotes the symbol of $\Psi$. Now since
\begin{equation*}
\prs{ \sigma [\Psi] (\psi) } = \xi_t \psi + \xi_x \psi\prs{\tfrac{x}{2t} },
\end{equation*}
then it follows that
\begin{equation*}
\left\langle \sigma \brk{\Psi}, \xi_t \right\rangle = \brs{\xi_t}^2 \neq 0.
\end{equation*}
Thus, by Holmgren's Uniqueness Theorem (cf. \cite{Taylor} pp.433) there exists
some $\epsilon >0$ such that on $\mathbb{R}^m \times \brk{s - \epsilon, s+
\epsilon}$, we have $\Psi(\psi) \equiv 0$. Therefore the set
\begin{equation*}
\mathcal{T} := \{ \theta \in \mathbb{R}_{\leq 0} : \psi_{\theta} = 0  \}
\end{equation*}
is open. Since this set is closed (it is $\Psi^{-1}( \{ 0 \})$ and $\Psi$ is
continuous) due to the connectedness of $\mathbb{R}_{\leq 0}$ we conclude that
$\mathcal{T} = (- \infty, 0)$, therefore $f_t \equiv \bar{f}_t$, as desired. The
result follows.
\end{proof}
\end{lemma}

\begin{lemma} \label{lem:char2}
A map $f : \mathbb{R}^m \to N$ is a soliton if and only if $f$ is a
harmonic map with respect to the metric $e^{\frac{-\brs{x}^2}{4(m-2)}}
g_{\Euc}$.
\begin{proof}
We begin with a general computation of how the tension field changes under a
conformal change of the base metric.  Take $g_{\ga \gb} := e^{2 \psi}
\delta_{\ga \gb}$ and then recall the formula for the conformally changed
Christoffel symbols
\begin{align*}
(\gG^{g})_{\ga \gb}^{\gz}
&=  - \delta^{\zeta \gw} (\del_w \psi) \delta_{\ga \gb} + (\del_{\ga} \psi)
\delta_{\gb}^{\gz} + (\del_{\gb} \psi) \delta_{\ga}^{\gz}.
\end{align*}
Since one has 
\begin{align*}
\tau_{g}^i &=  f^i_{\ga \ga} - \prs{\gG^{g}}_{\ga \ga}^{\gz} f^{i}_{\gz} +
f^k_{\ga} f^{j}_{\ga} \gG_{k j}^{i},
\end{align*}
it follows that
\begin{align*}
\tau_{g}^i &= \tau^i - \prs{  - \delta^{\zeta \gw} (\del_w \psi) \delta_{\ga
\ga} + (\del_{\ga} \psi) \delta_{\ga}^{\gz} + (\del_{\ga} \psi)
\delta_{\ga}^{\gz} }  f^i_{\gz} \\
&= \tau^i + \left( \del_{\gw}f^i \right)(\del_w \psi) m - (\del_{\ga} \psi)
\left( \del_{\ga} f^i\right) - (\del_{\ga} \psi) \left( \del_{\ga} f^i \right)
\\
&= \tau^i + (m - 2)( (\N \psi) \hook  Tf )^i.
\end{align*}
Choosing $\psi = \frac{-\brs{x}^2}{4(m-2)}$ the second term becomes
\begin{align*}
(m - 2)( \N \psi \hook Tf )&= - (m-2) \prs{\frac{x}{2(m-2)} \hook Tf} \\
&=- \prs{ \tfrac{x}{2} \hook Tf} .
\end{align*}
Applying this we see that
\begin{align*}
\tau_{g}^i
&= \tau^i - (\tfrac{x}{2} \hook Tf)^i.
\end{align*}
The result follows.
\end{proof}
\end{lemma}

\subsection{Preliminary identities}

In this subsection we establish some basic identities needed for the proof of
Theorem \ref{thm:main1}.  If we define the quantity
\begin{equation*}
G_{x_0,t_0}(x,t) :=  \frac{e^{-\frac{|x-x_0|^2}{4(t_0 - t)}}}{(4 \pi (t_0 -
t))^{\tfrac{m}{2}}}.
\end{equation*}
then $G_{x_0,t_0}$ is the heat kernel of $\mathbb{R}^m$ and thus
satisfies the backwards heat flow \eqref{eq:BHF}.  When the basepoint is
understood we will use the notation $G_0(x) = G_{x_0,t_0}(x,0)$.

\begin{defn}\label{def:peg}
Suppose $f : \mathbb{R}^m \to N$. Then $f$ has \emph{polynomial
energy density growth} if there exists a polynomial $p$ such that for all $x \in
\mathbb{R}^m$, one has $\brs{Tf(x)} \leq p(\brs{x})$.   Also, we say that $f $
has \emph{polynomial energy growth} if there exists a polynomial $p$ such that
\begin{equation*}
\int_{B_{\brs{x}}} \brs{Tf}^2 dV \leq p(\brs{x}).
\end{equation*}
\end{defn}
Note the assumption of polynomial energy growth needs to be made in various
lemmas below, to justify the applications of dominated convergence.  We first
prove a generalization of Lemma 3 of \cite{Zhang} with varying base
points for a $(\chi,\sigma)$-soliton $f$ and the heat kernel $G_0$.  We require
this generalization for the proof of Theorem \ref{thm:main1}.

\begin{lemma} \label{lem:primid} Let $f \in
\mathfrak{S}_{\chi,\sigma}$ satisfy polynomial energy growth, and let $\vp =
\vp^{\ga} \del_{\ga}$ be a vector
field on $M$ such that $|\vp|^2 G_0 \in L^{\infty}(\bRm)$. Then
\begin{align}
\begin{split}\label{eq:primid}
\int_{\bRm} \vp^{\ga} (x-x_0)^{\ga} | T f |^2 G_0 dV
& =  - 4 t_0 \int_{\bRm} \left\langle Tf_{\gb} , Tf_{\ga} \right\rangle
(\del_{\gb} \vp^{\ga} )  G_0 dV + 2 t_0 \int_{\bRm} ( \del_{\ga}\vp^{\ga} ) |T f
|^2 G_0 dV \\
& \hsp  + 2 \int_{\bRm} \left(  \frac{t_0}{\sigma} \chi  -  x_0 + x \left(1 - 
\frac{t_0}{\sigma}  \right)  \right)^{\gb} \left\langle Tf_{\gb} , Tf_{\ga}
\right\rangle \vp^{\ga} G_0 dV.
\end{split}
\end{align}
\begin{proof}
First observe that
\begin{equation}\label{eq:primidG}
\frac{\del G_0}{\del x^{\ga}} = - \frac{(x-x_0)^{\ga}}{2 t_0} G_0.
\end{equation}
Let $\eta \in C_c^{\infty}(\bRm)$ and $\vp =\vp^{\ga} \del_{\ga}$ be some smooth
vector field on $\bRm$. Then integrating by parts yields
\begin{align}
\begin{split}\label{eq:primid1}
\int_{\bRm} \vp^{\ga} (x-x_0)^{\ga} | Tf |^2 \eta G_0 dV
&= - 2 t_0 \int_{\bRm} \vp^{\ga} | Tf |^2 (\del_{\ga} G_0 ) \eta dV \\
& = 2 t_0 \int_{\bRm} \left( \del_{\ga}( \eta \vp^{\ga})| Tf |^2 +  \vp^{\ga}(
\del_{\ga} | Tf |^2)\eta \right) G_0 dV.
\end{split}
\end{align}
We consider the following manipulation, involving integration by parts once
more, in order to obtain an identity for the second term above.
\begin{align*}
4 t_0 \int_{\bRm} \left\langle \tau, Tf_{\ga} \right\rangle \vp^{\ga} \eta G_0
dV
& = 4 t_0 \int_{\bRm} \left\langle \N_{\gb} Tf_{\gb}, Tf_{\ga} \right\rangle
\vp^{\ga} \eta G_0 dV \\
&=4 t_0 \int_{\bRm} \left( \del_{\gb} \langle Tf_{\gb} , Tf_{\ga} \rangle -
\langle Tf_{\gb} , \N_{\gb}Tf_{\ga} \rangle \right) \vp^{\ga} \eta G_0 dV \\
&= - 4 t_0 \int_{\bRm}  \langle Tf_{\gb} , Tf_{\ga} \rangle \del_{\gb} \left[
\vp^{\ga} \eta G_0 \right]  dV  -  4 t_0 \int_{\bRm}  \langle Tf_{\gb} ,
\N_{\ga}Tf_{\gb} \rangle \vp^{\ga} \eta G_0 dV \\
&= -  4 t_0 \int_{\bRm} \left\langle Tf_{\gb} , Tf_{\ga} \right\rangle \left(
\del_{\gb} \lb \vp^{\ga}\eta \rb  G_0 - \frac{(x-x_0)^{\gb}}{2 t_0} \eta
\vp^{\ga} G_0\right) dV \\
& \hsp- 2 t_0 \int_{\bRm} \del_{\ga} | Tf |^2 \vp^{\ga}  \eta G_0dV.
\end{align*}
We therefore conclude that
\begin{align}
\begin{split}\label{eq:primids2}
 2 t_0 \int_{\bRm} \del_{\ga} | Tf |^2 \eta G_0dV
 &= - 4 t_0 \int_{\bRm} \left\langle Tf_{\gb} ,Tf_{\ga} \right\rangle \left(
\del_{\gb} \lb \vp^{\ga}\eta \rb  G_0 - \frac{(x-x_0)^{\gb}}{2 t_0} \eta
\vp^{\ga} G_0\right) dV \\
 & \hsp - 4 t_0 \int_{\bRm} \left\langle \tau, Tf_{\ga} \right\rangle \vp^{\ga}
\eta G_0 dV.
\end{split}
\end{align}
Applying this identity to \eqref{eq:primid1} we conclude that
\begin{align}
\begin{split}\label{eq:primidA}
\int_{\bRm} \vp^{\ga} (x-x_0)^{\ga} | Tf |^2 \eta G_0 dV
& =  - 4 t_0 \int_{\bRm} \left\langle Tf_{\gb} , Tf_{\ga} \right\rangle \left(
\del_{\gb} \lb \vp^{\ga}\eta \rb  G_0 - \frac{(x-x_0)^{\gb}}{2 t_0}\vp^{\ga} 
\eta G_0\right) dV\\
& \hsp  + 2 t_0 \int_{\bRm} \del_{\ga} \lb \vp^{\ga} \eta \rb | Tf |^2 G_0 dV
 - 4 t_0 \int_{\bRm} \left\langle \tau, Tf_{\ga} \right\rangle \vp^{\ga} \eta
G_0 dV.
\end{split}
\end{align}
We let $\eta_R$ be a cut off function with support within $B_R$ which cuts off
to zero linearly between $B_{R}$ and $B_{R+1}$. Setting $\eta = \eta_R$ above in
\eqref{eq:primidA} and sending $R \to \infty$, it follows from the Dominated
Convergence Theorem that we can remove $\eta$ from the quantities above. Since
$f \in \mathfrak{S}_{\chi,\sigma}$ we replace $\tau = \tfrac{(x- \chi)}{2
\sigma}$
accordingly
\begin{align*}
\int_{\bRm} \vp^{\ga} (x-x_0)^{\ga} | Tf |^2 G_0 dV
& =  - 4 t_0 \int_{\bRm} \left\langle Tf_{\gb} , Tf_{\ga} \right\rangle
(\del_{\gb} \vp^{\ga} )  G_0 dV + 2 t_0 \int_{\bRm} (\del_{\ga} \vp^{\ga}) | Tf
|^2 G_0 dV \\
& \hsp  + 2 \int_{\bRm} \left(  \frac{t_0}{\sigma} \chi  -  x_0 + x \left(1 - 
\frac{t_0}{\sigma}  \right)  \right)^{\gb} \left\langle Tf_{\gb} , Tf_{\ga}
\right\rangle \vp^{\ga} G_0 dV.
\end{align*}
The result follows.
\end{proof}
\end{lemma}

\begin{cor} \label{cor:solitonids} Suppose that $f \in
\mathfrak{S}_{\chi,\gs}$. Then the following identities hold.
\begin{enumerate}
\item[(a)] 
\begin{align*}
\int_{\bRm} |x-x_0|^2 | Tf |^2 G_0 dV =&\ 2 \int_{\bRm} \left( |(x-x_0) \hook Tf
|^2 - \frac{t_0}{\sigma} \left\langle (x-
\chi)\hook Tf, (x-x_0)\hook Tf \right\rangle \right) G_0 dV\\
& \hsp +  2 t_0 \int_{\bRm}\left( m - 2 \right) | Tf |^2 G_0 dV.
\end{align*}
\item[(b)]
\begin{align*}
\int_{\bRm}  (x-x_0)^{\gamma} | Tf |^2 G_0 dV
& = 2 \int_{\bRm}  \left\langle \left(  \frac{t_0}{\sigma} \chi  -  x_0 + x
\left(1 -  \frac{t_0}{\sigma}  \right) \hook Tf  \right), T f_{\gamma}
\right\rangle  G_0 dV.
\end{align*}
\end{enumerate}

\begin{proof}
Identity (a) follows by setting $\vp^{\ga} := (x-x_0)^{\ga}$, while identity (b)
follows by setting $\vp^{\ga} := \delta^{\ga \gamma}$.
\end{proof}
\end{cor}

The following corollary is a result of coinciding base points of the heat kernel
$G_0$ and $f$, and follows from choice of appropriate vector fields.

\begin{cor}[\cite{Zhang}, Lemma 3]\label{cor:lem3zhang} Suppose that $f \in
\mathfrak{S}$, $\gamma \in [1,m] \cap \mathbb{N}$, and $\zeta = \zeta^{\ga}
\del_{\ga}$ is some constant vector field on $\bRm$. Then the following
identities hold.
\begin{enumerate}
\item[(a)] $\int_{\bRm} \left( \frac{2-m}{4} + \frac{|x-x_0|^2}{8 t_0} \right) |
Tf |^2 G_0 dV = 0$,
\item[(b)] $\int_{\bRm} (x-x_0)^{\gamma} |Tf |^2 G_0 dV=0$,
\item[(c)] $\int_{\bRm} |x-x_0|^4 | Tf |^2 G_0 dV = \int_{\bRm} \left( 4m(m-2)
t_0^2  | Tf |^2 - 32 t_0^3 | \tau |^2 \right) G_0 dV$,
\item[(d)] $\int_{\bRm} |x-x_0|^2 \langle \zeta, x-x_0 \rangle |Tf|^2 G_0 dV =
\int_{\bRm} \langle \zeta \hook Tf, \tau \rangle G_0 dV = 0$,
\item[(e)] $\int_{\bRm} \langle \zeta, x-x_0 \rangle^2 | Tf |^2 G_0 dV = 2
t_0\int_{\bRm} \left( |\zeta|^2 |Tf |^2 - 2 | \zeta \hook Tf |^2 \right) G_0
dV$.
\end{enumerate}

\begin{proof}
The lemma follows by choosing various test functions, indicated below.
\begin{enumerate}
\item[(a)] This follows by setting $\vp^{\ga} := \frac{(x-x_0)^{\ga}}{8 t_0}$.
\item [(b)] This follows by setting $\vp^{\ga} := \delta^{\ga \gamma}$.
\item [(c)] We set $\vp^{\ga} := |x-x_0|^2 (x-x_0)^{\ga}$. Then we have that
\begin{align*}
\del_{\gb} \vp^{\ga} = 2 (x-x_0)^{\gb} (x-x_0)^{\ga} + \delta^{\ga}_{\gb}
|x-x_0|^2.
\end{align*}
and obtain
\begin{align*}
\int_{\bRm} |x-x_0|^4 | Tf |^2 G_0 dV
& =  - 8 t_0 \int_{\bRm} \brs{Tf(x-x_0)}^2 G_0 dV  - 4 t_0 \int_{\bRm} | Tf |^2
|x-x_0|^2 G_0 dV\\
& \hsp + 2 t_0 \int_{\bRm} \left((2+ m) |x-x_0|^2 \right) | Tf |^2 G_0 dV \\
&= - 32 t_0^3 \int_{\bRm} |\tau|^2 G_0 dV  + 2m t_0 \int_{\bRm} \left( |x-x_0|^2
\right) | Tf |^2 G_0 dV.
\end{align*}
applying part $(a)$ to the second term yields
\begin{equation*}
\int_{\bRm} |x-x_0|^4 | Tf |^2 G_0 dV  =- 32 t_0^3 \int_{\bRm} |\tau|^2 G_0 dV+
4m (m-2) t_0^2 \int_{\bRm} | Tf |^2 G_0 dV.
\end{equation*}

\item[(d)] First setting $\vp^{\ga} := |x-x_0|^2 \zeta^{\ga}$ and applying (b)
we have
\begin{align}
\begin{split}\label{eq:lem3zhangd1}
\int_{\bRm} |x-x_0|^2 \langle (x-x_0),\zeta \rangle | Tf |^2 G_0 dV
& =  - 8 t_0 \int_{\bRm} \left\langle   (x-x_0) \hook Tf , \zeta \hook Tf
\right\rangle  G_0
dV.
\end{split}
\end{align}
Next we will consider $\vp^{\ga} := \langle \zeta, x-x_0 \rangle (x-x_0)^{\ga}$.
Note that
\begin{equation*}
\del_{\gb}\lb \zeta^{\mu} (x-x_0)^{\mu} (x-x_0)^{\ga} \rb = \zeta^{\mu}
\delta^{\mu}_{\gb} (x-x_0)^{\ga} + \zeta^{\mu} (x-x_0)^{\mu} \delta_{\gb}^{\ga}
= \zeta^{\gb} (x-x_0)^{\ga} + \left\langle \zeta , x-x_0 \right\rangle
\delta_{\gb}^{\ga}.
\end{equation*}
Applying this identity, and using (b) once more we obtain
\begin{align}
\begin{split}\label{eq:lem3zhangd2}
\int_{\bRm} |x-x_0|^2 \langle (x-x_0),\zeta \rangle | Tf |^2 G_0 dV
& =  - 4 t_0 \int_{\bRm} \left\langle Tf_{\gb} , Tf_{\ga} \right\rangle (
\zeta^{\gb} (x-x_0)^{\ga} + \left\langle \zeta , x-x_0 \right\rangle
\delta_{\gb}^{\ga})  G_0 dV \\
& \hsp + 2 t_0 \int_{\bRm} ((1+m)\left\langle \zeta, (x-x_0) \right\rangle) | Tf
|^2 G_0 dV\\
& =  - 4 t_0 \int_{\bRm} \left\langle  \zeta \hook Tf , (x-x_0) \hook Tf
\right\rangle
G_0 dV.
\end{split}
\end{align}
Comparing equalities \eqref{eq:lem3zhangd1} and \eqref{eq:lem3zhangd2} we
conclude the desired result.

\item[(e)] This follows by setting $\vp^{\ga} := \langle \zeta , x-x_0 \rangle
\zeta^{\ga}$.
\end{enumerate}
The results follow.
\end{proof}
\end{cor}

We next observe that as a consequence of the above identites, the harmonic map
heat flow with source of
dimension $m=2$ cannot exhibit type I singularities.

\begin{prop} Let $f_t: M^m \to N$ be a smooth solution to harmonic map
heat flow which exists on a maximal time interval $[0,t_0)$ with $t_0<\infty$.
If $m=2$ then
\begin{equation*}
\lim_{t \to t_0}(t_0-t) \brs{Tf_t} = \infty.
\end{equation*}
Moreover, any soliton on $\mathbb{R}^m$ for $m \leq 2$ is constant.

\begin{proof}
Suppose to the contrary there exists some $C \in \mathbb{R}$ such that $\lim_{t
\to t_0} (t_0-t) \brs{Tf_t} \leq C$. By \cite{Grayson}) one may construct a
type I blowup limit $f_{\infty}$ which is a self similar solution and thus its
time $t=0$ slice with $t_0 = 1$ is a non constant soliton. By Corollary
\ref{cor:lem3zhang} (a), since $m = 2$ one has that
\begin{equation*}
\int_{\bRm} \left( \frac{|x|^2}{8} \right) | Tf |^2 G_0 dV = 0.
\end{equation*}
Therefore $f_{\infty}$ is constant, but this is a contradiction. Thus, the
result follows.
\end{proof}
\end{prop}

\subsection{Variation identities}

For the following computations we will consider the quantities $f$, $x$, $t$ to
be varying by a parameter $s$.  Often notational dependency on $s$ will be
suppressed except for on the base point variables $(x_s, t_s)$. We will first
lay out some fundamental variation identities which will be applied to the
formulas of Propositions \ref{prop:1stvarM} and \ref{prop:2ndvarM} to
demonstrate the identities specific to the Euclidean source case. We then
evaluate at a soliton in order to derive the $L^f$ operator, a quantity whose
negative spectrum characterizes the stability of its corresponding soliton $f$.

\begin{lemma} \label{lem:Gderivs}
The following identities hold.
\begin{align*}
\frac{d G_s}{d s} &= \mathsf{g}_s G_s, \\
\frac{d^2 G_s}{d s^2} &= \prs{ \frac{ d \mathsf{g}_s}{ds} } G_s + \prs{
\mathsf{g}_s }^2 G_s,
\end{align*}
where 
\begin{equation}\label{eq:defsfg}
\mathsf{g}_s :=  \left( -  \frac{m \dot{t}_s}{2 t_s} + \frac{\dot{t}_s
|x-x_s|^2}{4 t_s^2} + \frac{\langle \dot{x}_s , x-x_s \rangle}{2 t_s} \right).
\end{equation}

\begin{proof}
Using chain rule, and observing that $G_s = \tfrac{e^{-
\frac{\brs{x_s}^2}{4t_s}}}{(4 \pi t)^{m/2}}$ and applying \ref{eq:primidG} we
have
\begin{align*}
\frac{d G_s}{d s} 
&= \frac{\del G_s}{\del x^{\ga}} \frac{\del x^{\ga}}{\del s}  + \frac{\del
G_s}{\del t} \frac{\del t}{\del s}\\
&=  - \frac{\ip{ x- x_s, \dot{x}_s }}{2 t_s} \frac{e^{-
\tfrac{\brs{x-x_s}^2}{4t_s}}}{(4 \pi t_s)^{m/2}} + \prs{- \frac{\dot{t}_s}{4
t_s^2}} \frac{e^{ -\tfrac{\brs{x-x_s}^2}{4t_s}}}{(4 \pi t_s)^{m/2}} - \frac{m
\dot{t}_s}{2} \frac{e^{-\tfrac{\brs{x-x_s}^2}{4t_s}}}{(4 \pi)^{m/2}
t_s^{m/2+1}}\\
&= \left( -  \frac{m \dot{t}_s}{2 t_s} + \frac{\dot{t}_s |x-x_s|^2}{4 t_s^2} +
\frac{\langle \dot{x}_s , x-x_s \rangle}{2 t_s} \right) G_s \\
&= \mathsf{g}_s G_s.
\end{align*}
This gives the first identity. Differentiating again we have that
\begin{align*}
\frac{d^2 G_s}{d s^2} 
&= \frac{d}{ds} \brk{ \mathsf{g}_s G_s} \\
&= \prs{ \frac{d \mathsf{g}_s}{ds}  + \mathsf{g}_s^2 }G_s. 
\end{align*}
The result follows.
\end{proof}
\end{lemma}

Next we compute some identities operations performed on $\mathsf{g}_s$ (cf.
\eqref{eq:defsfg}) which will be used for the following variations.
\begin{align}
\frac{\del \mathsf{g}_s}{\del x^{\ga}} &= \left( \frac{\dot{t}_s}{2 t_s}
(x-x_s)^{\ga} + \frac{\dot{x}_s^{\ga}}{2t_s} \right) \\
\begin{split}\label{eq:delgid}
\frac{d \mathsf{g}_s}{d s}  &= - \frac{m}{2} \left( \frac{\ddot{t}_s}{t_s} -
\frac{\dot{t}_s^2}{t_s^2} \right) + \frac{|x-x_s|^2}{4} \left(
\frac{\ddot{t}_s}{t^2_s} - \frac{2 \dot{t}_s^2}{t_s^3} \right) - \frac{\dot{t}_s
\langle \dot{x}_s, x-x_s \rangle}{t^2_s} + \frac{\langle \ddot{x}_s , x-x_s
\rangle}{2t_s} - \frac{|\dot{x}_s|^2}{2t_s}
\end{split} \\
\begin{split}
\mathsf{g}^2_s &= \left( - \frac{m \dot{t}_s}{2 t_s}+ \frac{\dot{t}_s |x -
x_s|^2}{4 t_s^2} + \frac{\langle \dot{x}_s, x - x_s \rangle}{2 t_s} \right)^2\\
&= \frac{m^2 \dot{t}_s^2}{4 t_s^2} - \frac{m \dot{t}_s^2 |x-x_s|^2}{4 t_s^3} -
\frac{m \dot{t}_s \langle \dot{x}_s , x-x_s \rangle}{2 t_s^2} + \frac{\dot{t}_s
|x-x_s|^2 \langle \dot{x}_s, x-x_s \rangle}{4 t_s^3} \\
& \hsp + \frac{\dot{t}_s^2 |x-x_s|^4}{16 t_s^4} + \frac{\langle \dot{x}_s, x -
x_s \rangle^2}{4 t_s^2}. \label{eq:g2id}
\end{split}
\end{align}

\begin{cor}[\cite{Zhang} Proposition 3]\label{prop:1stvar}
Let $f_s$, $x_s$, $t_s$ denote the variations of $f$, $x_0$, and $t_0$
respectively, and set
\begin{equation*}
\dot{t}_s = \frac{d t_s}{ds}, \ \dot{x}_s = \frac{d x_s}{ds}, \ \dot{f}_s =
\frac{d f_s}{ds}.
\end{equation*}
Assuming that $f_s \in H_{loc}^1(\bRm, N)$ and $\dot{f}_s$ satisfies the
integrability condition
\begin{equation*}
\int_{\bRm} \left( |\dot{f}_s |^2 + | \N \dot{f}_s |^2 + |x-x_s|^2 | T f_s |^2 +
| \tau_s |^2 \right) G_s dV < +\infty.
\end{equation*}
Then
\begin{gather}\label{eq:1stvar}
\begin{split}
\frac{d}{ds} \lb \mathcal{F}_{x_s, t_s} (f_s) \rb
& = - t_s \int_{\bRm} \left\langle \dot{f}_s, \mathsf{S}_{x_s,t_s}(f_s)
\right\rangle G_s dV \\
& \hsp + \dot{t}_s \int_{\bRm} \left( \frac{2 -m}{4} + \frac{|x-x_s|^2}{8 t_s}
\right) | T f_s |^2 G_s dV\\
& \hsp + \frac{1}{4} \int_{\bRm} \left\langle \dot{x}_s, x-x_s \right\rangle |T
f_s|^2 G_s dV.
\end{split}
\end{gather}

\begin{proof} By appealing to Proposition \ref{prop:1stvarM} and apply the
identities of Lemma \ref{lem:Gderivs} we obtain
\begin{align*}
\tfrac{d}{ds}\brk{ \mathcal{F}_{x_s,t_s}(f_s) }
&= \tfrac{1}{2} \int_{\mathbb{R}^m} \prs{ \dot{t}_s + t_s \mathsf{g}_s } \brs{T
f_s}^2  G_s dV - t_s \int_{\mathbb{R}^m} \ip{ \dot{f}_s,
\mathsf{S}_{x_s,t_s}(f_s) } G_s dV.
\end{align*}
Simply applying the definition of $\mathsf{g}_s$, \eqref{eq:defsfg}, and
rearranging terms gives the result.
\end{proof}
\end{cor}

One can conclude the following statement from this formulation combined with the
identities (a) and (b) of Corollary \ref{cor:lem3zhang} applied to
\eqref{eq:1stvar}.

\begin{cor}[\cite{Zhang} Corollaries 1, 2]\label{cor:critpoints} The point
$(\N, x_0, t_0)$ is a critical point of the $\mathcal{F}$-functional if and only
if $\N$ is an $(x_0,t_0)$-soliton.
\end{cor}

\begin{cor} \label{lem:2ndvar}
Assuming that $f_s \in H_{loc}^1(\bRm \times I, N)$ and for all $s \in I$,
$\dot{f}_s$ satisfies the integrability condition
\begin{equation*}
\int_{\bRm} \left( |\dot{f}_s |^2 + | \N \dot{f}_s |^2 + |x-x_s|^2 | T f_s |^2 +
| \tau_s |^2 \right) G_s dV < + \infty.
\end{equation*}
Then
\begin{align*}
\frac{d^2}{ds^2}\brk{\mathcal{F}_{x_s,t_s}(f_s)}
&=- t_s \int_{\bRm} \ip{ \dot{f}_s, \lap \dot{f}_s +
R^N(\dot{f},(Tf_s)_{\ga})(Tf_s)_{\ga}) - \N_{\frac{x-x_s}{2 t_s}} \dot{f}_s }
G_s dV\\
& \hsp  - \int_{\bRm} \ip{ \dot{f}, (2 t_s \mathsf{g}+\dot{t}_s)
\mathsf{S}_{x_s,t_s}(f_s) + \prs{\dot{t}_s(x-x_s)+\dot{x} } \hook  Tf_s  } G_s
dV\\
& \hsp + \tfrac{1}{2} \int_{\bRm} \prs{ t_s \prs{\dot{\mathsf{g}_s} +
\mathsf{g}^2_s} + 2 \dot{t}_s \mathsf{g} + \ddot{t}_s } \brs{Tf_s}^2 G_s dV_g -
t_s\int_{\bRm} \ip{ \ddot{f}_s, \mathsf{S}_{x_s,t_s}(f_s) } G_s dV.
\end{align*}
\begin{proof}
We start by applying the identities of Lemma \ref{lem:Gderivs} and
\eqref{eq:defsfg} directly to the formula of Proposition \ref{prop:2ndvarM}.
\begin{align}
\begin{split}
\tfrac{d^2}{ds^2}\brk{\mathcal{F}_{x_s,t_s}(f_s)}
&=\frac{\ddot{t}_s}{2}\int_{\bRm} |Tf_s|^2 G_s dV- t_s\int \ip{ \ddot{f}_s,
\mathsf{S}_{x_s,t_s}(f_s) } G_s dV+\frac{t_s}{2}\int_{\bRm} |Tf_s|^2\prs{
\dot{\mathsf{g}}_s + \mathsf{g}^2_s} G_s dV\\
&\hsp -\dot{t}_s\int_{\bRm} \prs{\ip{ \dot{f}_s,\mathsf{S}_{x_s,t_s}(f)}
-|Tf_s|^2 \mathsf{g}_s  } G_s dV\\
& \hsp -t_s\int_{\bRm} \ip{ \dot{f}_s,(\lap
\dot{f}_s+R^N(\dot{f}_s,(Tf_s)_{\ga})(Tf_s)_{\ga}) - \N_{\frac{x-x_s}{2t_s}}
\dot{f}_s } G_s dV \\
& \hsp - 2 t_s \int_{\bRm} \ip{ \dot{f}_s, \tau_s \mathsf{g}_s +\prs{\frac{\N
\prs{\mathsf{g}_s G_s}}{G_s} \hook Tf_s} } G_s dV.
\end{split}
\end{align}
We expand the quantity of the last line.
\begin{align*}
 \int_{\bRm} \ip{ \dot{f},\tau \mathsf{g} +  \prs{\tfrac{\N\prs{\mathsf{g}
G}}{G}} \hook Tf } G_s dV
&= \int_{\mathbb{R}^m} \ip{ \dot{f}, \tau \mathsf{g} + Tf \prs{ (\N \mathsf{g})
+ \mathsf{g} \prs{\tfrac{\N G}{G}}}} G_s dV \\
&= \int_{\mathbb{R}^m} \ip{ \dot{f}, \tau \mathsf{g} + \prs{ \tfrac{\dot{t}_s}{2
t_s}(x-x_s) + \tfrac{\dot{x}_s}{2 t_s} - \mathsf{g} \prs{\tfrac{x-x_s}{2 t_s} }}
\hook Tf} G_s dV \\
&= \int_{\mathbb{R}^m} \ip{ \dot{f}, \mathsf{g} \mathsf{S}_{x_s,t_s}(f) +
\tfrac{1}{2 t_s} \prs{\dot{t}_s(x-x_s)+\dot{x} }\ \hook Tf  } G_s dV.
\end{align*}
Then we regroup terms and obtain the result.
\end{proof}
\end{cor}

Evaluating the previous variation at a soliton allows us to conclude the formula
for the $L^f$ operator, which will be key in the following computations. This
will be given by
\begin{gather} \label{Lfdef}
 \begin{split}
L^f &: \gG \prs{ f^*(TN) } \to \gG \prs{ f^*(TN) } \\
& : X \mapsto \left[ -  \lap X - R (X, Tf_{\ga}) Tf_{\ga} + \frac{\N_{(x-x_0)}
X}{2 t_0}\right].  
 \end{split}
\end{gather}
We also set
\begin{equation} \label{Wdef}
W_f^{2,2} := \left\{ X \in \gG(f^* TN) : \int_{\mathbb{R}^m} \prs{ \brs{X}^2 +
\brs{\N X}^2 + \brs{L^f X}^2 }G_0 dV_g \right\}.
\end{equation}

\begin{prop}[\cite{Zhang} Proposition 5]\label{prop:2ndvarsol} Let $f$ be a
soliton and $\dot{f} =
\dot{f}^{\ga} \del_{\ga}$ be some vector field on $N$ such that $\dot{f} \in
W_f^{2,2}$.
Let $f_s$ be some one parameter family with $f_0 = f$, and $\dot{t_0}$ and
$\dot{x_0}$ basepoint variations. Then
\begin{align*}
\mathcal{F}''(\dot{t},\dot{x},\dot{f}) &= \frac{d^2}{ds^2} \left. \lb
\mathcal{F}_{x_s, t_s}(f_s) \rb \right|_{s=0} \\
&=  t_0 \int_{\bRm} \left\langle \dot{f}, L^f \dot{f} - 2 \dot{t}_0 \tau -
(\dot{x}_0 \hook T f)  \right\rangle G_0 dV - \int_{\bRm}  \left( \dot{t}_0^2
\left| \tau  \right|^2 + \frac{1}{2}
\left| \dot{x}_0 \hook T f  \right|^2 \right) G _0 dV.
\end{align*}

\begin{proof}
Since $\mathsf{S}_{x_0,t_0}(f) = 0$ and thus $Tf_{\ga} (x-x_0)^{\ga} =2 t_0
\tau$,
\begin{align}
\begin{split}\label{eq:2ndvarsol1}
\left. \frac{d^2}{ds^2} \lb \mathcal{F}_{x_s,t_s} (f) \rb \right|_{t=0}
&= - t_s \int_{\bRm} \left\langle \dot{f} , \lap \dot{f}    +     R(\dot{f}, T
f_{\ga}) T f_{\ga}  - \frac{1}{2 t_0}  \N_{(x-x_0)} \dot{f} \right\rangle G dV 
\\
& \hsp -  \int_{\bRm}{ \left\langle \dot{f}, 2 \dot{t}_0 \tau  + \dot{x}_0 \hook
T f \right\rangle G dV} \\
& \hsp   + \frac{1}{2} \int_{\bRm} \left( t_0(\dot{\mathsf{g}} + \mathsf{g}^2) +
\ddot{t}_0 + 2 \dot{t}_0 \mathsf{g} \right) | Tf|^2 G dV.
\end{split}
\end{align}

We expand, marking with the applicable identities of Corollary
\ref{cor:solitonids},
\begin{align*}
\dot{\mathsf{g}} + \mathsf{g}^2&=
\left(  - \frac{m \ddot{t}_0}{2 t_0} + \frac{m \dot{t}_0^2}{2 t_0^2} + \frac{m^2
\dot{t}_0^2}{4 t_0^2} \right) + \left( \frac{|x-x_0|^2}{4 t_0} \left(
\frac{\ddot{t}_0}{t_0} - \frac{m \dot{t}_0^2}{t_0^2} - \frac{2
\dot{t}_0^2}{t_0^2} \right) \right)_a + \left( \langle \dot{x}_0, x-x_0 \rangle
\left( - \frac{\dot{t}_0}{t^2_0} - \frac{m \dot{t}_0}{2 t_0^2} \right) \right)_b
\\
& \hsp  + \left( |x-x_0|^2 \langle \dot{x}_0, x-x_0 \rangle \left(
\frac{\dot{t}_0}{4 t_0^3} \right) \right)_d + \left( \langle \ddot{x}_0, x-x_0
\rangle \frac{1}{2t_0} \right)_b + \left( |\dot{x}_0|^2 \left( \frac{-1}{2t_0}
\right) \right)\\
& \hsp + \left( |x-x_0|^4 \left( \frac{\dot{t}_0^2 }{16 t_0^4} \right) \right)_c
+ \left( \langle \dot{x}_0, x - x_0 \rangle^2 \left(  \frac{1}{4 t_0^2} \right)
\right)_e.
\end{align*}
Therefore we have
\begin{align*}
\int_{\bRm}{ \left( \dot{\mathsf{g}} + \mathsf{g}^2 \right)| Tf|^2 G dV}
&=  \int_{\bRm} \left(  - \frac{m \ddot{t}_0}{2 t_0} + \frac{m \dot{t}_0^2}{2
t_0^2} + \frac{m^2 \dot{t}_0^2}{4 t_0^2} \right) |Tf |^2 G dV  \\
& \hsp +  \int_{\bRm}\left( \frac{(m-2)}{2} \left( \frac{\ddot{t}_0}{t_0} - 
\frac{m \dot{t}_0^2}{t_0^2} - \frac{2 \dot{t}_0^2}{t_0^2} \right) \right) |Tf
|^2 G dV  \\
& \hsp   - \int_{\bRm}   \left( \frac{|\dot{x}_0|^2}{2t_0} \right)  |Tf |^2 G dV
+\left(  \frac{1}{4 t_0^2} \right) 2 t_0 \int_{\bRm} \left( |\dot{x}_0|^2 |Tf
|^2 - 2 \brs{ \dot{x}_0 \hook T f }^2 \right) G dV\\
& \hsp +  \left( \frac{\dot{t}_0^2 }{16 t_0^2} \right) \left( \int_{\bRm} \left(
4m(m-2)  | Tf |^2 - 32 t_0 | \tau |^2 \right) G dV \right) \\
&= \int_{\bRm}\left(  \left(  - \frac{m \ddot{t}_0}{2 t_0} + \frac{m
\dot{t}_0^2}{2 t_0^2} + \frac{m^2 \dot{t}_0^2}{4 t_0^2} \right) +
\frac{(m-2)}{2} \left( \frac{\ddot{t}_0}{t_0} -  \frac{m \dot{t}_0^2}{t_0^2} -
\frac{2 \dot{t}_0^2}{t_0^2} \right) \right) | Tf |^2 G dV\\
& \hsp + \int_{\bRm}{\left(m (m-2) \frac{\dot{t}_0^2 }{4 t_0^2} \right)  | Tf
|^2 G dV} \\
& \hsp + \int_{\bRm}\left( -\frac{2 \dot{t}_0^2}{ t_0}|\tau|^2 -\frac{1}{ 
t_0}\left( \brs{  \dot{x}_0 \hook T f }^2 \right)\right)G dV.
\end{align*}
We combine the non spatial dependent coefficients in the integrand of the form
$\int_{\bRm}{( \cdot ) | Tf |^2 G_0 dV}$ and obtain
\begin{align*}
\left(- \frac{\ddot{t}_0}{t_0} + \frac{2 \dot{t}_0^2}{t_0^2}  \right) + m \left(
\frac{- \ddot{t}_0}{2 t_0} + \frac{\dot{t}_0^2}{2t_0^2} +
\frac{\ddot{t}_0}{2t_0} - \frac{ \dot{t}_0^2}{t_0^2} +
\frac{\dot{t}_0^2}{t_0^2}- \frac{\dot{t}_0^2}{2 t_0^2} \right) + m^2 \left(
\frac{\dot{t}_0^2}{4 t_0^2} - \frac{\dot{t}_0^2}{2 t_0^2} + \frac{\dot{t}_0^2}{4
t_0^2} \right).
\end{align*}
Thus the coefficients in front of $m$ and $m^2$ vanish accordingly. Therefore we
conclude that
\begin{align*}
\frac{t_0}{2} \int_{\bRm} (\dot{\mathsf{g}} + \mathsf{g}^2) | Tf|^2 G dV =
\int_{\bRm} \left(- \frac{\ddot{t}_0}{2} + \frac{\dot{t}_0^2}{t_0}  \right)
|Tf|^2 G dV  - \int_{\bRm}\left( \dot{t}_0^2|\tau|^2 + \frac{1}{2}\left( \brs{ 
\dot{x}_0 \hook T f}^2 \right)\right)G dV.
\end{align*}
Considering the last line of \eqref{eq:2ndvarsol1} and incorporating identities
(a) and (b) of Corollary \ref{cor:lem3zhang} yields
\begin{align*}
\frac{1}{2}\int_{\bRm} (t_0 (\dot{\mathsf{g}} + \mathsf{g}^2) + \ddot{t}_0 +
2\dot{t}_0 \mathsf{g}) | Tf |^2 G dV &=  \int_{\bRm} \left(
-\frac{\ddot{t}_0}{2} + \frac{\dot{t}_0^2}{t_0} + \frac{\ddot{t}_0}{2}
+\dot{t}_0 \mathsf{g}_0 \right) | Tf |^2 G dV \\
& \hsp - \int_{\bRm}\left( \dot{t}_0^2|\tau|^2 +\frac{1}{2}\left( \brs{ 
\dot{x}_0 \hook T f}^2 \right)\right)G dV \\
&=  \int_{\bRm} \left( \frac{(2- m) \dot{t}_0^2}{2 t_0} + \frac{\dot{t}_0^2
|x-x_0|^2}{4 t_0^2} + \dot{t}_0 \frac{\langle \dot{x}_0 , x-x_0 \rangle}{2 t_0} 
\right) | Tf |^2 G dV \\
& \hsp - \int_{\bRm}\left(  \dot{t}_0^2|\tau|^2 +\frac{1}{2}\left( \brs{ 
\dot{x}_0 \hook T f}^2 \right)\right)G dV \\
&=  - \int_{\bRm}\left(  \dot{t}_0^2|\tau|^2 +\frac{1}{2}\brs{  \dot{x}_0 \hook
T f}^2  \right)G dV.
\end{align*}
Combining everything we conclude the result.
\end{proof}
\end{prop}

\section{Entropy stability and
\texorpdfstring{$\FF$}{F}-stability}\label{s:entropystability}

In this section we prove Theorem
\ref{thm:main1}. In analogy to \cite{CM}, this shows the equivalence of $\calf$
and entropy stability for self shrinkers provided they do not split off a line
isometrically. We characterize these notions in the following definitions.

\begin{defn} \label{defn:entropstab}
A soliton $f \in \mathfrak{S}$ is \emph{entropy stable} if it is a local
minimizer of the entropy functional $\lambda$.
\end{defn}

\begin{defn} \label{def:Fstable} A soliton $f \in \mathfrak{S}$
is called $\mathcal{F}$-stable if for any $X \in W_f^{2,2}$ there exists a real
number $q \in \mathbb{R}$ and a constant vector field $V \in T\mathbb{R}^m$ such
that $\mathcal{F}{''}(q,V,X) \geq 0$, (where $\mathcal{F}''$ is as in
Proposition \ref{prop:2ndvarsol}).
\end{defn}

\begin{defn}\label{def:cylindrical} We say a map $f$ is \emph{cylindrical} if
there is a constant
vector field $\zeta \in T\mathbb{R}^m$ such that
\begin{equation*}
\zeta \hook T f  \equiv 0.
\end{equation*}
\end{defn}

We build up the proof of Theorem \ref{thm:main1} in a series of lemmas.  First,
in Lemma \ref{lem:harsoltnconst} we show that a soliton which is also a harmonic
map is necessarily constant.  Using this we show in Lemma \ref{lem:lem7.10CM}
that the point $(0,1) \in \mathbb R^m \times \mathbb R_{\geq 0}$ realizing the
supremum in the definition of $\gl$ is a \emph{strict} maximum as long as the
soliton is not cylindrical.  Moreover there is an effective estimate of the drop
in the $\FF$-functional provided the basepoint is taken a small definite
distance from $(0,1)$.  Using this estimate and an observation that the
$\FF$-functional is strictly convex along carefully chosen paths in $\mathbb R^m
\times \mathbb R_{\geq 0}$, we obtain the proof of Theorem \ref{thm:main1}.

\begin{lemma}\label{lem:harsoltnconst} Suppose $f \in \mathfrak{S}$ has
polynomial energy density growth and satisfies $\tau = 0$.  Then $f$ is a
constant map.
\end{lemma}

\begin{proof}
Suppose to the contrary that $f$ is a nonconstant harmonic soliton. The
corresponding self similar solution with the given $f$ as the time $-1$ slice,
as described in Lemma \ref{lem:char1}, satisfies that for all $\kappa \in
\mathbb{R}$, we have $f(x,t) = f(\kappa x,\kappa^2 t)$. Thus
\begin{equation*}
Tf_t(x) = \kappa \prs{Tf_{\kappa^2 t}(\kappa x)},
\end{equation*}
Note that
because $f$ is nontrivial there exists some $y \in \bRm$ at which the following
limit holds.
\begin{equation*}\label{eq:limofF}
\lim_{t \to 0} \brs{Tf_t \left( y \sqrt{-t}\right)} = \lim_{t \to 0}
\tfrac{1}{\sqrt{-t}} \brs{
Tf_{-1} \left( y\right) }= \infty.
\end{equation*}
In particular, $\sup_{\mathbb{R}^m \times [-1,0)} \brs{ Tf_t}
= \infty$.

Simultaneously since $\tau= 0$ and solutions to the
\eqref{eq:HMHF} on $\mathbb R^m$ with fixed initial conditions subject to the
polynomial growth condition are unique, we obtain that $\dt f_t = 0$ for
all $(x,t) \in \mathbb R^n \times [-1,0)$.  Thus for all $x \in \bRm$, we have
$\brs{Tf_t } = \brs{Tf_{-1}}$. This implies that $\sup_{x \in \mathbb{R}^n,{t
\in [-1,0)}} \brs{ Tf_t } = \sup_{x \in
\mathbb{R}^n} \brs{Tf_{-1}}  < \infty$, which is a contradiction. The result
follows.
\end{proof}

\begin{defn}\label{def:Xi} Given a one-parameter family of smooth maps $f_s \in
C^{\infty}(\bRm, N)$, let
\begin{equation*}
\Xi : \bRm \times \mathbb{R}_{\geq 0} \times I : (x,t,s) \mapsto
\mathcal{F}_{x,t}(f_s),
\end{equation*}
and moreover set $\Xi (x,t) := \Xi(x,t,0)$.
\end{defn}

\begin{lemma} \label{lem:lem7.10CM} Suppose that $f \in
\mathfrak{S}$ has polynomial energy density growth and is noncylindrical. Given
$\epsilon >0$ there exists $\delta_{\epsilon} > 0$ such that
\begin{equation}
\sup_{\{(x_0,t_0) : |x_0 | + |\log t_0 | > \epsilon \}} \mathcal{F}_{x_0,t_0}(f)
 < \la (f)
- \delta_{\epsilon}.
\end{equation}
\begin{proof}
The argument follows closely work in \cite{CM}. The proof, which demonstrates
that $(0,1)$ is a strict global
maximum, is divided into two steps. \textbf{Step 1} involves demonstrating that
$\Xi$ has a strict local maximum at $(0,1)$, by first characterizing $(0,1)$ as
a critical point and then checking its concavity. For \textbf{Step 2} we show
that $\Xi$ decreases along a family of paths spanning over the entire space-time
domain $\mathbb{R}^m \times \mathbb{R}_{\geq 0}$ which emanate from $(0,1)$.
\newline

\noindent \textbf{Step 1.} Recalling the first variation formula of Proposition
\ref{prop:1stvar}, since $f$ is a soliton then by Corollary \ref{cor:critpoints}
the gradient of $\Xi$ vanishes at $(0,1)$, thus $f$ is a critical point. Next if
we consider the second variation formula of Proposition
\ref{prop:2ndvarsol} applied to a fixed $f$ and evaluated along a path
$(sy,1+sh)$ for $s>0$ and $h \in \mathbb{R}$ we obtain
\begin{equation}\label{eq:lem7.10CM1}
\frac{d^2}{d s^2} \lb \Xi(sy,1+sh) \rb = - \int_{\bRm}  \left( h^2 \left| \tau_s
 \right|^2 + \frac{1}{2} \left| y \hook  T f_s \right|^2 \right) G_s dV.
\end{equation}
Note that \eqref{eq:lem7.10CM1} is nonpositive. The first quantity on the right
vanishes only if $h=0$ or if $f_s$ is harmonic and therefore trivial by Lemma
\ref{lem:harsoltnconst}. The second quantity on the right only vanishes if
$y \hook Tf_s = 0$, which forces $y = 0$ since $f$ is noncylindrical. Therefore
$\Xi$ has a strict local maximum at $(0,1)$. \newline

\noindent \textbf{Step 2.} We next show that for a given $y \in \bRm$ and $h \in
\mathbb{R}$, one has $\frac{d}{d s}\lb \Xi (sy, 1 + h s^2) \rb \leq 0$ for
all $s>0$ such that $1+h s^2 > 0$. Using Proposition \ref{prop:1stvar} and
replacing $x_0 \mapsto x_s$ and $t_0 \mapsto t_s$ and $G_0 \mapsto G_s$,
imposing $\tfrac{\del f_s}{\del s} \equiv 0$, and then inserting identities (a)
and (b) of Corollary \ref{cor:solitonids} we obtain
\begin{align*}
\frac{d}{ds} \lb \Xi(x_s, t_s)\rb 
 &= \int_{\bRm}   \left(    \dot{t}_s \left( \frac{2- m}{4} + \frac{
|x-x_s|^2}{8 t_s} \right) + \frac{\langle \dot{x}_s , x-x_s \rangle}{4} \right)
G_s dV \\
& =  \frac{\dot{t}_s}{8 t_s } 2 \int_{\bRm} \left( \brs{(x-x_s) \hook Tf_s }^2 -
t_s\left\langle  x \hook Tf_s,(x-x_s) \hook  Tf_s  \right\rangle \right) G_s
dV\\
& \hsp +  \frac{1}{2} \int_{\bRm}  \left\langle   \left(  -  x_s + x \left(1 - 
t_s  \right)  \right) \hook Tf_s , \dot{x}_s  \hook T f_s \right\rangle  G_s dV.
\end{align*}
Then
\begin{align*}
\left. \frac{d}{ds} \lb \Xi(x_s, t_s)\rb \right|_{x_s = sy, t_s = 1 + hs^2}
& =  \frac{2hs}{8 (1+hs^2) } 2 \int_{\bRm} \left( \brs{ (x-sy) \hook Tf_s }^2 -
t_s\left\langle x \hook Tf_s, (x-sy) \hook Tf_s \right\rangle \right) G_s dV\\
& \hsp  -  \frac{1}{2} \int_{\bRm}  \left\langle  \left(   sy +  x  hs^2 
\right)\hook Tf_s   , y \hook T f_s \right\rangle  G_s dV \\
& =  \frac{hs}{2 (1+hs^2) } \int_{\bRm} \left( |(x-sy) \hook Tf_s |^2 -
(1+hs^2)\left\langle x \hook Tf_s ,(x-sy) \hook Tf_s \right\rangle \right) G_s
dV\\
& \hsp  +  \frac{s}{2} \frac{h}{(1+hs^2)} \int_{\bRm}  \left\langle
\frac{(1+hs^2)}{h} \prs{\left(   -y -  x  hs  \right) \hook Tf_s } ,  y \hook T
f_s \right\rangle  G_s dV \\
&= - \frac{s}{2(1+hs^2)} \int_{\bRm} \left( s^2 h^2 | x \hook Tf_s  |^2  + 2sh
\left\langle x \hook Tf_s  , y \hook Tf_s \right\rangle + | y \hook  Tf_s |^2
\right)
G_s dV \\
&= - \frac{s}{(1+hs^2)} \int_{\bRm} \left| (shx + y) \hook Tf_s\right|^2 G_s dV.
\end{align*}
Since $(1+hs^2) \geq 0$ for all of the paths parametrized by $(sy,1+h s^2)$ we
have that the derivative of $\Xi$ is nonpositive on the union of these paths. We
conclude the desired result.
\end{proof}
\end{lemma}

\begin{lemma}\label{lem:Fdecaylemma} Let $f \in C^{\infty}(\bRm)$ have
polynomial growth. Then for all $x_0 \in \bRn$,
\begin{equation}\label{eq:Fdecaylemma}
\lim_{t_0 \to 0} \int_{\mathbb{R}^m}{f G_{x_0,t_0} dV} = f(x_0).
\end{equation}
Furthermore given a map $f : \mathbb R^m \to (N, h)$ with polynomial energy
density growth, then for all $x_0 \in \bRm$,
\begin{equation*}
\lim_{t_0 \to 0} \mathcal{F}_{x_0,t_0}(f) = 0.
\end{equation*}

\begin{proof} We have that \eqref{eq:Fdecaylemma} is a well known trait of the
heat kernel. We address the second identity. Since $\brs{T f}^2$ has
polynomial growth we have
\begin{align*}
\lim_{t_0 \to 0} \FF_{x_0,t_0}(f) =&\ \lim_{t_0 \to 0} t_0 \int_{\mathbb R^n}
\brs{Tf}^2 G_0 = \left( \lim_{t_0 \to 0} t_0 \right) \left( \lim_{t_0 \to
0} \int_{\mathbb R^n} \brs{Tf}^2 G_0 \right) = 0,
\end{align*}
as required.
\end{proof}
\end{lemma}

\begin{proof}[Proof of Theorem \ref{thm:main1}]
If $f$ is not $\mathcal{F}$-stable then there exists a variation $f_s$ for $s
\in
[- 2 \epsilon, 2 \epsilon]$ with $f_0 = f$ satisfying the following
\begin{enumerate}
\item[(V1)]\label{V1} For each variation $f_s$ of $f$, the support of $f_s -
f$ is compact.
\item[(V2)]\label{V2} For any paths $(x_s,t_s)$ with $x_0 = 0$ and $t_0 =1$,
\begin{equation}\label{eq:thm0.15CM1}
\left. \frac{d^2}{d s^2} \lbr \mathcal{F}_{x_s,t_s}(f_s) \rbr
\right|_{s=0} < 0.
\end{equation}
\end{enumerate}

For the one-parameter family $f_s$, we set $\Xi$ to be as in Definition
\ref{def:Xi}.  Also, set
\begin{equation}
B^{\circ}_r:= \{ (x,t,s) : 0 < |x| + \brs{\log t} + s < r \}.
\end{equation}
We claim that there exists $\epsilon' > 0$ so that for $s
\neq 0$ and $|s| \leq \epsilon'$ one has
\begin{equation}\label{eq:thm0.15CM2}
\la (f_s) := \sup_{x_0,t_0 } \Xi(x_0,t_0,s) < \Xi(0,1,0) = \la(f).
\end{equation}
Following \cite{CM} we proceed in five steps:%
\begin{enumerate}
\item[(1)] $\Xi$ has a strict local maximum at $(0,1,0)$.
\item[(2)] $\Xi(\cdot,\cdot,0)$ has a strict global maximum at $(0,1,0)$.
\item[(3)] $\frac{\del}{\del s} \lbr \Xi(x_0,t_0,s)\rbr$ is uniformly bounded on
compact sets.
\item[(4)] For $|x_0|$ sufficiently large, $\Xi(x_0,t_0,s) < \Xi(0,1,0)$.
\item[(5)] For $| \log t_0 |$ sufficiently large, $\Xi(x_0,t_0,s) < \Xi(0,1,0)$.
\end{enumerate}
These steps yield the groundwork for the desired result at the end of the proof.

\noindent \textbf{Proof of (1):}
Since $f$ is a soliton, by Corollary \ref{cor:critpoints}, given a path
$(x_s,t_s)$ with
$(x_0,t_0) = (0,1)$ and a variation $f_s$ of $f$, we have
$\left. \frac{\del}{\del s}\lbr \Xi(x_s,t_s, s) \rbr \right|_{s=0} = 0$, which
implies that $(0,1,0)$ is a critical point of $\Xi$. Consider one such path of
the form $(sy,1+hs)$ for $y \in \mathbb R^n$, $h \in \mathbb{R}$ and some
variation of $f$ given by $f_{bs}$ for some $b \neq 0$. Then we have that, by
property (V2),
\begin{equation*}
\left. \frac{\del^2}{\del s^2} \lbr \Xi(sy,1+hs,bs)  \rbr \right|_{s=0}=b^2
\left. \frac{\del^2}{\del s^2} \lbr \mathcal{F}_{x_s,y_s}(f_{s})  \rbr
\right|_{s=0} \leq 0,
\end{equation*}
where here $x_s = s \frac{y}{b}$ and $t_s = 1 + \frac{h}{b}s$.
Now we consider the second variation when $b = 0$. As an immediate application
of Proposition \ref{lem:lem7.10CM} we have that $\left. \frac{\del^2}{\del s^2}
\lbr
\Xi(sy,1+hs,0) \rbr \right|_{s=0} < 0$. Therefore the Hessian of
$\Xi$ is negative definite at $(0,1,0)$, and thus $\Xi$ attains a strict local
maximum at this point. We may choose $\epsilon' \in (0,\epsilon)$ such that for
$(x_0,t_0,s) \in B^{\circ}_{\epsilon'}$ we have that
\begin{equation*}
\Xi(x_0,t_0,s) < \Xi(0,1).
\end{equation*}

\noindent \textbf{Proof of (2):} This is an immediate result of Proposition
\ref{lem:lem7.10CM}. Therefore, $\la(f) = \Xi(0,1)$ and we may choose $\delta>0$
so
that for all points of the form $(x_0,t_0,0)$ outside $B^{\circ}_{{\epsilon' /
4}}$
we have that
\begin{equation*}
\Xi(x_0,t_0) < \Xi(0,1) - \delta.
\end{equation*}

\noindent \textbf{Proof of (3):} Using Proposition \ref{prop:1stvar}, we see
that
\begin{align*}
\frac{\del}{\del s} \lbr \Xi(x_0,t_0,s) \rbr
&= - t_s \int_{\bRm} \left\langle \dot{f}_s, \mathsf{S}_{x_0,t_0}(f_s)
\right\rangle G_0 dV .
\end{align*}
Using the assumption of polynomial energy density growth and applying the
dominated convergence theorem it follows that $\tfrac{\del \Xi}{\del s}$,
is continuous in all three variables
$x_0$, $t_0$ and $s$. Therefore $\tfrac{\del \Xi}{\del s}$ is uniformly bounded
on compact sets. \\
\noindent \textbf{Proof of (4):}  By hypothesis, we may choose $R \in
\mathbb{R}_{>0}$ so that
the support of $f - f_s$ is contained in $B_R \subset \bRm$. Let $\rho > 0$
and consider $|x_0| > \rho + R$. Then we have that
\begin{align*}
\Xi(x_0,t_0,s)
&= t_0 \int_{\mathbb{R}^m} \left| Tf_s \right|^2 G_{0} dV \\
& = t_0 \int_{B_R} \left| T f_s \right|^2 G_{0} dV +  t_0
\int_{\mathbb{R}^m \backslash B_R} \left| T f \right|^2 G_{0} dV \\
& \leq t_0^{\frac{2-m}{2}} (4 \pi)^{- \frac{m}{2}} \int_{B_R} \left| T f_s
\right|^2 e^{-\frac{|x-x_0|^2}{4 t_0}} dV +  \Xi(x_0,t_0,0) \\
&\leq t_0^{\frac{2-m}{2}} (4 \pi)^{- \frac{m}{2}} e^{-\frac{\rho^2}{4
t_0}}\int_{B_R} \left| T f_s \right|^2  dV +  \Xi(x_0,t_0,0).
\end{align*}
By compactness of the domain $B_R \times [-2 \epsilon , 2 \epsilon ]$ we know
that $\int_{B_R} \left| T f_{s} \right|^2  dV < C_R$ for some $C_R \in
\mathbb{R}$. Therefore we conclude that
\begin{equation}\label{eq:sc4Xiest}
\Xi(x_0,t_0,s) \leq (4 \pi)^{-\frac{m}{2}} C_R t_0^{\frac{2-m}{2}}
e^{-\frac{\rho^2}{4 t_0}} + \Xi(x_0,t_0,0).
\end{equation}
Define the quantity
\begin{equation}\label{eq:sc4muRdef}
\mu_{\rho}(\tau) := \tau^{\frac{2-m}{2}} e^{- \frac{\rho^2}{4 \tau}}.
\end{equation}
We note in particular that
\begin{equation*}
\mu_1 \left( \frac{\tau}{\rho^2} \right) =\left( \frac{\tau}{\rho^2}
\right)^{\frac{2-m}{2}} e^{- \frac{\rho^2}{4 \tau}} = \rho^{m-2}
\mu_{\rho}(\tau).
\end{equation*}
The function $\mu_1$ is clearly continuous and therefore bounded and also
satisfies the following limit for $\ga \in \{ 0 ,\infty\}$,
\begin{align*}
\lim_{\tau \to \ga} \mu_1 \left( \tau \right) = \lim_{\tau \to \ga}
\tau^{\frac{2-m}{2}} e^{-\frac{1}{4 \tau }} = 0.
\end{align*}
We thus conclude that
\begin{equation}
\lim_{\rho \to \infty} \left( \sup_{\tau> 0} \mu_{\rho}(\tau) \right) =
\lim_{\rho \to \infty} \sup_{\tau> 0}\left(   \rho^{2-m} \mu_1 \left(
\frac{\tau}{\rho^2} \right) \right) = 0.
\end{equation}
Therefore, as a consequence of (2) combined with this above limit, we conclude
that for $|x_0|$ sufficiently large we have that $\Xi(x_0,t_0,s) < \Xi(0,1,0)$,
as desired.

\noindent \textbf{Proof of (5):} We first perform the following manipulation,
for $R \in \mathbb{R}_{>0}$
\begin{gather}
\begin{split}\label{eq:s5Ximanipulation}
\Xi(x_0,t_0,s)
&= t_0 \int_{\mathbb{R}^m} \left| T f_s \right|^2 G_{0} dV \\
& = t_0 \int_{B_R} \left| T f_s \right|^2 G_{0} dV +  t_0
\int_{\mathbb{R}^m \backslash B_R} \left| T f_s \right|^2 G_{0} dV \\
& \leq t_0^{\frac{2-m}{2}} (4 \pi)^{- \frac{m}{2}} \int_{B_R} \left| T f_s
\right|^2 G_0 dV +  \Xi(x_0,t_0,0) \\
& \leq C_R t_0^{\frac{2-m}{2}} (4 \pi)^{- \frac{m}{2}} +  \Xi(x_0,t_0,0).
\end{split}
\end{gather}
As a result of this, we also obtain the estimate
\begin{equation*}
\sup_{t_0 \geq 1} \Xi(x_0,t_0,s) \leq C_R (4 \pi)^{-\frac{m}{2}} + \la (f).
\end{equation*}
We break into two cases.  First, suppose $t_0$ is very large.   Combining
\eqref{eq:s5Ximanipulation} with part (2) we obtain the claim.  The case when
$t_0$ is small, in particular $t_0 \leq 1$, is more difficult.  Using
Proposition \ref{prop:1stvar} with $\dot{t}_0 = 1$, we have
\begin{gather} \label{thm1derbnd}
\begin{split}
\frac{\del}{\del t_0} \prs{ \Xi(x_0, t_0, s) }
&= \int_{\mathbb{R}^m} \left[ \frac{2 -m}{4} + \frac{|x-x_0|^2}{8 t_0} \right]
| T f_s |^2 G_0 dV\\
\geq&\ - C_0(\brs{x_0}).
\end{split}
\end{gather}
Note that $C_0$ is independent of $x_0$, $t_0$, and $s$ subject to the
restriction $|x_0|<R$.  Recall from Step (2), that
\begin{equation}
\Xi(0,1,0) = \la(f) > 0.
\end{equation}
Choose $\ga > 0$ so that $3 \ga < \la(f)$, and choose $t_{\ga} =
\frac{\ga}{C_0}$.
For any $x \in \bRm$ and $s \in [-\epsilon, \epsilon]$, by Lemma
\ref{lem:Fdecaylemma} there exists some $t_{x,s} > 0$ such that for all $t_0
\leq
t_{x,s}$ we have $|\Xi(x,t_0,s)|<\ga$.

On the set $\overline{B_{R+1}} \times [-\epsilon, \epsilon]$ we will construct a
finite open cover as follows. The cover consists of balls $b_i$ of radius
$r_i>0$ centered at $(x_i,t_i)$. Each $b_i$ has an associated time $ t_i  \leq
\min \{ t_{\ga}, 1 \}$ where
\begin{enumerate}
\item Given $(x,s)$ there exists and index $i(x,s)$ such that $(x,s) \in
b_{i(x,s)}$.
\item For each $b_i$ the associated $t_i$ is such that
\begin{equation*}
\left. \Xi(x,t_i,s) \right|_{b_i} < \ga.
\end{equation*}
Note that this choice follows from the existence of $t_{x,s}$ and the continuity
of $\Xi$.
\end{enumerate}
Choosing a finite subcover of the $b_i$'s we let $\bar{t}$ be the minimum of all
corresponding $t_i$. Then as a result of the derivative estimate
\ref{thm1derbnd} we
have that for any triple $(x,t_0,s)$ with $s \in [-\epsilon,\epsilon]$, and $x
\in \overline{B_{R}}$, and $t_0 \leq \bar{t}$,
\begin{align*}
\Xi(x,t_0,s)   \leq \Xi(x, t_{i(x,s)},s) + C_0 \left(t_{i(x,s)} - t_0 \right)
\leq 2\ga  < \la(f).
\end{align*}
Claim (5) follows.\\

Given claims (1)-(5) we finish the proof by dividing the domain into regions
corresponding to the size of $|x_0| + \brs{\log t_0}$.  Using (1), when $s$ is
sufficiently small there exists some $r > 0$ such that $\Xi(x_0,t_0,s) <
\Xi(x_0,t_0,0)$ for $(x_0,t_0)$ within the following region
\begin{equation*}
\mathfrak{R}_1 := \{ (x_0,t_0) : |x_0| + |\log t_0| < r \}.
\end{equation*}
Using (4) and (5) there exists an $R>0$ such that $\Xi(x_0,t_0,s) <
\Xi(x_0,t_0,0)$ for $(x_0,t_0)$ in the following region.
\begin{equation*}
\mathfrak{R}_2 := \{ (x_0,t_0) : |x_0| + |\log t_0| > R \}.
\end{equation*}
Therefore it remains to consider
\begin{equation}
\mathfrak{R}_3 := \{ (x_0,t_0) : R > |x_0| + |\log t_0| > r \}.
\end{equation}
Given $(x_0,t_0) \in \mathfrak{R}_3$, we know by (2) that $\Xi(x_0,t_0,0) <
\la(f)$, and by (3) that the $s$ derivative of $\Xi$ is uniformly bounded. So
we may choose a $\delta > 0$ such that $\Xi$ restricted to the region
$\mathfrak{R}_3 \times [-\delta,\delta]$ is bounded above by $\la(f)$.
Therefore, \eqref{eq:thm0.15CM2} holds on $\bigcup_{i=1}^3\mathfrak{R}_i$ and
as this union constitutes the entire space-time domain, the result follows.

Therefore we conclude that entropy stability implies $\mathcal{F}$-stability.
For
the converse, assume that $f$ is $\mathcal{F}$-stable. Then by definition for
any variation $f_s$ of $f$ there exists some $\sigma \in
\mathbb{R}$ and $\zeta \in TM$ such that $\mathcal{F}''(\dot{f}_0, \sigma,
\zeta)  \geq 0$. Hence for families $x_s$, $t_s$ with $\dot{x}_s =
\zeta$ and $\dot{t}_s = \sigma$ and sufficiently small $s$ we have
\begin{align*}
\la(f_s) &\geq \mathcal{F}_{x_s, t_s}(f_s) \\
&\geq \mathcal{F}_{0,1}(f) \\
&= \la (f).
\end{align*}
Therefore we have that $\la(f_s) \geq \la(f)$ for any variation $f_s$ of $f$,
and the result follows.
\end{proof}
\section{Rigidity Results} \label{s:stabilityrigidity}

In this section we establish rigidity results for solitons, focusing primarily
on the case of spherical targets.  Our main goal is to establish Theorem
\ref{thm:main2}.  To accomplish this we first show Theorem \ref{thm:rayleigh},
which shows that showing a negative upper bound for the Rayleigh quotient
characterizing the lowest eigenmode of $L^f$ suffices to establish instability
of a given soliton.  With this in place, by carefully exploiting conformal
vector fields on the sphere we establish such an upper bound in the presence of
large entropy, finishing the proof of Theorem \ref{thm:main2}.

\subsection{\texorpdfstring{$\mathcal{F}$}{F}-stability and Eigenvector fields
of \texorpdfstring{$L^f$}{Lf}}

For the rest of the paper we only consider solitons based at $(0,1)$, and drop
the basepoint on the Greens function in our notation.  Given the equivalence of
entropy stability and $\calf$-stability, we will study
spectral properties of the $L^f$ operator, which for a soliton based at $(0,1)$
is
\begin{equation*}
L^fV=-\Delta V - R(V,Tf_\alpha)Tf_\alpha + \N_{\frac{x}{2}}V.
\end{equation*}
Our first lemma concerns the action of $L^f$ on the pushforward of a vector
field.
\begin{lemma}\label{lem:Lzeta}
Suppose $f \in C^{\infty}(\mathbb{R}^m,N)$ is a soliton. Given a vector field
$\zeta^{\ga} \del_{\ga} \in T \mathbb{R}^m$,
\begin{align*}
L^f \left(  \zeta \hook Tf\right)   =  - \prs{ \frac{\zeta}{2} \hook Tf} -  2
(\N_{\ga} \zeta^{\gb})(\N_{\ga} \del_{\gb} f) -Tf(\lap \zeta) +  \left(
\N_{\frac{x}{2}} \zeta \right) \hook Tf.
\end{align*}

\begin{proof}
Recall that by the definition of $L^f$,
\begin{align*}
L^f \left( \zeta \hook Tf\right)^j &= - \N_{\gb} \N_{\gb} \lb \zeta^{\gamma}
Tf^j_{\gamma} \rb -
( \del_{\ga} f^{q})( \del_{\ga} f^j) (\dot{f}^p) R_{pqj}^{i} \zeta^{\eta}
Tf^i_{\eta} - \frac{x^{\ga}}{2} \N_{\ga} (\zeta^{\eta} Tf^j_{\eta}).
\end{align*}
We first compute
\begin{align*}
\lap \left( \zeta^{\gb} (f^i_{\gb})\right)
&= \N_{\ga} \N_{\ga} \left( \zeta^{\gb} f^i_{\gb} \right) \\
&= \N_{\ga} \lb \zeta^{\gb} \N_{\ga} f^i_{\gb} \rb + \N_{\ga} \left(\N_{\ga}
\zeta^{\gb} f^i_{\gb} \right) \\
&=\zeta^{\gb} (\N_{\ga} \N_{\ga}  f^{i}_{\gb} ) +  2 (\N_{\ga}
\zeta^{\gb})(\N_{\ga}  f^i_{\gb} ) + (\N_{\ga} \N_{\ga} \zeta^{\gb})  f^i_{\gb}.
\end{align*}
We expand the first term:
\begin{align*}
\zeta^{\gb} (\N_{\ga} \N_{\ga} \left(  f^{i}_{\gb} \right)) &= \zeta^{\gb}
(\N_{\ga} \N_{\gb} \left(  f^{i}_{\ga} \right)) \\
&=  \zeta^{\gb} ( [\N_{\ga}, \N_{\gb}]  f^{i}_{\ga} )  + \zeta^{\gb} ( \N_{\gb}
\N_{\ga}  f^{i}_{\ga} )\\
&=  \zeta^{\gb} ( (\del_{\ga} f^p) (\del_{\gb} f^q )R_{pqk}^{i}  f^{k}_{\ga})  +
\zeta^{\gb} ( \N_{\gb} \N_{\ga}  f^{i}_{\ga} ).
\end{align*}
We apply the soliton equation to the first term
\begin{align*}
\zeta^{\gb} \N_{\gb} \tau^i
&= \zeta^{\gb} \N_{\gb} \left( \frac{x^{\ga}}{2} \del_{\ga} f^i \right) \\
&=  \frac{\zeta^{\ga}}{2} (\del_{\ga} f^i) +  \zeta^{\gb} \frac{x^{\ga}}{2}
\N_{\gb} (\del_{\ga} f^i ) \\
&=  \frac{\zeta^{\ga}}{2} (\del_{\ga} f^i) +  \zeta^{\gb} \frac{x^{\ga}}{2}
\N_{\gb} (\del_{\ga} f^i ).
\end{align*}
Therefore
\begin{align*}
\lap \left( \zeta^{\gb} (\del_{\gb} f^i)\right) &= \zeta^{\gb} ( (\del_{\ga}
f^p) (\del_{\gb} f^q )R_{pqk}^{i} \left( \del_{\ga} f^{k} \right))  + 
\frac{\zeta^{\ga}}{2} (\del_{\ga} f^i) \\
& \hsp +  \zeta^{\gb} \frac{x^{\ga}}{2} \N_{\gb} (\del_{\ga} f^i ) +  2
(\N_{\ga} \zeta^{\gb})(\N_{\ga} \del_{\gb} f^i) + (\N_{\ga} \N_{\ga}
\zeta^{\gb}) (\del_{\gb} f^i).
\end{align*}
Therefore
\begin{align*}
L^f(\zeta \hook Tf )^i &= - \frac{1}{2} \zeta^{\ga} (\del_\ga f^i) - 
\zeta^{\gb} \frac{x^{\ga}}{2} \N_{\gb} (\del_{\ga} f^i ) -  2 (\N_{\ga}
\zeta^{\gb})(\N_{\ga} \del_{\gb} f^i) - (\N_{\ga} \N_{\ga} \zeta^{\gb})
(\del_{\gb} f^i) + \frac{x^{\gb}}{2} \N_{\gb} \lb \zeta^{\ga} \del_{\ga} f^i
\rb\\
&= - \frac{1}{2} \zeta^{\ga} (\del_\ga f^i) -  2 (\N_{\ga} \zeta^{\gb})(\N_{\ga}
\del_{\gb} f^i)  - (\N_{\ga} \N_{\ga} \zeta^{\gb}) (\del_{\gb} f^i) +
\frac{x^{\gb}}{2} (\N_{\gb} \zeta^{\ga})( \del_{\ga} f^i ).
\end{align*}
Thus we conclude that
\begin{align*}
L^f ( \zeta \hook Tf )^i =- \frac{1}{2} \zeta^{\ga} (\del_\ga f^i) -  2
(\N_{\ga} \zeta^{\gb})(\N_{\ga} \del_{\gb} f^i)  - (\N_{\ga} \N_{\ga}
\zeta^{\gb}) (\del_{\gb} f^i) + \frac{x^{\gb}}{2} (\N_{\gb} \zeta^{\ga})(
\del_{\ga} f^i ).
\end{align*}
The result follows.
\end{proof}
\end{lemma}

\begin{cor}[\cite{Zhang} Proposition 6]
Suppose $\zeta^i \del_i$ is a constant vector field. Then we have that
\begin{align*}
L^f \left(\zeta \hook Tf\right)   =  - \frac{1}{2} \zeta \hook Tf.
\end{align*}
Furthermore, 
\begin{align*}
L^f\prs{\frac{x}{2} \hook Tf}=-\frac{x}{2} \hook TF.
\end{align*}
\end{cor}
\begin{cor} The following equality holds
\begin{equation*}
\left\langle L^f(\zeta \hook Tf ), (\zeta \hook Tf ) \right\rangle =-
\frac{1}{2} \brs{\zeta \hook Tf  }^2 -  \left\langle  \prs{\lap \zeta - \left(
\N_{\tfrac{x}{2}} \zeta \right)} \hook Tf  +  2 (\N_{\ga} \zeta^{\gb})(\N_{\ga}
Tf_{\gb}), \zeta \hook Tf \right\rangle.
\end{equation*}
\begin{proof} We simply compute using the formula of Lemma \ref{lem:Lzeta}.
\begin{align*}
\left\langle L^f(\zeta \hook Tf  ), (\zeta \hook Tf ) \right\rangle
&= - \frac{1}{2} | \zeta \hook Tf  |^2 - 2 \N_{\ga} \zeta^{\gb} \langle \N_{\ga}
Tf_{\gb}, \zeta  \hook Tf  \rangle - \langle \prs{\lap \zeta}  \hook Tf, \zeta 
\hook Tf \rangle\\
& \hsp + \left\langle  \left( \N_{\tfrac{x}{2} } \zeta \right)  \hook Tf, \zeta
\hook Tf\right\rangle \\
&=- \frac{1}{2} | \zeta \hook Tf  |^2 -  \left\langle  \prs{\lap \zeta}  \hook
Tf +  2 \N_{\ga} \zeta^{\gb}\N_{\ga} Tf_{\gb} -  Tf\left( \N_{\tfrac{x}{2}}
\zeta \right) , \zeta \hook Tf \right\rangle\\
&=- \frac{1}{2} \brs{\zeta \hook Tf  }^2 -  \left\langle  \prs{\lap \zeta -
\left( \N_{\tfrac{x}{2}} \zeta \right)} \hook Tf  +  2 (\N_{\ga}
\zeta^{\gb})(\N_{\ga} Tf_{\gb}), \zeta \hook Tf \right\rangle.
\end{align*}
The result follows.
\end{proof}
\end{cor}
Next we give a generalization of (\cite{Zhang} Theorem 3).
\begin{prop}\label{prop:gapthm}
Suppose $f : \bRm \rightarrow N$ is a soliton and the sectional curvature
of $N$ is bounded above by $K \in \mathbb{R}_{>0}$. If $K|Tf|^2\leq 1$ then $f$
is constant. In
particular, if the sectional curvature of $N$ is nonpositive then any soliton
whose target is $N$ must be constant.
\begin{proof} Integrating by parts, we get
\begin{align*}
\int_{\bRm} | \N \tau |^2 G_0 dV
&=\int_{\bRm} \left\langle \N_{\ga} \tau, \N_{\ga} \tau \right\rangle G_0 dV \\
&= - \int_{\bRm} \left\langle  \tau, \lap \tau \right\rangle G_0  dV_g +
\int_{\bRm}\left\langle \frac{x}{2} \hook \N \tau, \tau \right\rangle G_0 dV
\\
&= \int_{\bRm} \left\langle - \lap \tau +  \frac{x}{2} \hook \N \tau  , \tau
\right\rangle G_0 dV\\
&= \int_{\bRm} \left\langle L^f \tau  + R(\tau, Tf_{\ga}) Tf_{\ga}, \tau
\right\rangle G_0 dV\\
&= \int_{\bRm} \left\langle - \tau  + R(\tau, Tf_{\ga}) Tf_{\ga}, \tau
\right\rangle G_0 dV\\
& =\int_{\bRm}\left( -| \tau |^2 + \left\langle R(\tau, Tf_{\ga})Tf_{\ga} , \tau
\right\rangle \right) G_0 dV \\
&\leq  \int_{\bRm} | \tau |^2 \left(K|Tf|^2-1\right) G_0 dV.
\end{align*}
Therefore $\tau$ is parallel and $|\tau|$ is constant. By the soliton equation
and using polar coordinates,
\begin{equation*}
\tau = \prs{ \frac{x}{2}  \hook Tf} = \frac{1}{2} r \del_r f.
\end{equation*}
Therefore $|\tau|^2 = \tfrac{r^2}{4} \brk{\del_r f}^2 \in \mathbb{R}$. Since
this is true for all $r$, we must have that $\del_r f = \tau  = 0$. Thus $f$ is
a constant map.
\end{proof}
\end{prop}
\begin{defn} \label{mudefn}
Let $f:\bRm\to N$ be a soliton. The bottom of the spectrum of $L^f$ is defined
to be
\begin{equation*}
\mu_1=\inf_V\frac{\int_{\bRm} \langle L^fV,V\rangle G_0 dV}{\int_{\bRm} |V|^2G_0
dV},
\end{equation*}
taken over smooth vector fields $V\in C^\infty_0(f^{*}TN)$ of compact support. 
\end{defn}
We note that this infimum can be taken over Lipschitz vector fields $V$ in the
space $W^{2,2}_f$ as defined in \ref{Wdef}. At first glance there is
nothing stopping this infimum from being $-\infty$, but Theorem
\ref{thm:rayleigh}, which we now prove, shows that if $\mu_1$ is sufficiently
negative, then the soliton in question is unstable. 
\begin{proof}[Proof of Theorem 1.2]
We follow closely the
corresponding Lemma 9.58 in \cite{CM}.  Let $\mu_1(R)$ denote the smallest
Dirichlet eigenvalue of $L^f$ in a ball of radius $R$ centered at the origin. 
Observe that $\mu_1(R)$ is decreasing and $\lim_{R\to\infty}\mu_1(R)=\mu$. 
Choose some $R>0$ big enough such that $\mu_1(R)<-\frac{3}{2}$ and let $V$ be a
nonzero Dirichlet eigenfield $L^fV=\mu_1(R)V$ on the ball $B_R$. Let $h \in
\mathbb{R}$, $y\in\bRm$, and consider the second variation identity for $\calf$
in the directions determined by $V$, $h$ and $y$ as in \ref{prop:2ndvarsol},
giving
\begin{align*}
\FF''(h,y,V) = \int_{B_R}\prs{\langle V,L^f V-2h\tau-(y\hook
Tf)\rangle-h^2|\tau|^2-\frac{1}{2}|y\hook Tf|^2} G_0 dV.
\end{align*}
We apply Young's inequality to the term $-\langle V,y\hook Tf\rangle\leq
\frac{1}{2}(|V|^2+|y\hook Tf|^2)$ and find that this is bounded above by
\begin{align*}
\int_{B_R}\prs{\prs{\frac{1}{2}+\mu_1(R)}|V|^2-2h\langle
V,\tau\rangle-t^2|\tau|^2 }G_0
dV\leq \int_{B_R}-|V|^2-2h \prs{\langle V,\tau\rangle-h^2|\tau|^2 }G_0 dV < 0,
\end{align*}
for all choices of $h$ and $y$. It follows that $f$ is $\calf$-unstable.
\end{proof}

We note that if one could prove the result that if $\mu_1>-\infty$ then there is
a vector field $V\in W^{2,2}$ such that $L^fV=\mu_1 V$, a modification of the
previous proof could show that $\mu_1<-1$ implies $\calf$ instability.

\subsection{Solitons with Spherical Targets}

In this section we will restrict our discussion to solitons $f:\mathbb{R}^m\to
S^n$ with spherical target. In fact, in light of Proposition \ref{prop:gapthm},
there are no nontrivial solitons without some positive sectional curvature on
the target. Our restriction is therefore natural because the sphere is the
simplest manifold for which a finite time singularity in the harmonic map heat
flow can occur.

An interesting observation about the harmonic map heat flow from a compact
manifold into the sphere $S^n$, $n\geq 3$, is that any harmonic map $f:M\to S^n$
is unstable or constant. This is because the second variation of the energy at
such a harmonic map satisfies
\begin{align*}
\frac{d^2}{dt^2} \brk{\mathcal{E}(f_t) }=\int_{\bRm}|\N
V|^2-|Tf|^2|V|^2+|\langle
V,Tf\rangle|^2 dV,
\end{align*}
which, remarkably, is negative when summed over an appropriate basis of
conformal vector fields. Through a similar calculation, we will study the
$\mathcal{F}$-stability of a soliton $f:\mathbb{R}^m\to S^n$ and its
relationship to conformal vector fields. Recall that a conformal vector field
$W$ on the sphere $S^n$ is constructed by fixing some $w\in\mathbb{R}^{n+1}$ and
considering the vector field whose value at $p\in S^n$ is
\begin{align*}
W_p:=w-\langle w,p\rangle p,
\end{align*}
the projection of the constant vector field $w$ to the tangent space at $p$.
We record a preliminary lemma on these vector fields.
\begin{lemma}\label{lemma:confsum}
Let $w_0,w_1,\ldots,w_n$ be any orthonormal basis of $\mathbb{R}^{n+1}$. Let
$W_i$ be the corresponding conformal vector fields on $S^n$ and let $v$ be any
vector tangent to the sphere at some point $p$.
Then
\begin{enumerate}
\item $\sum_{i=0}^n|W_i|^2=n$,
\item $\sum_{i=0}^n\langle W_i,v\rangle ^2=|v|^2$.
\end{enumerate}
\begin{proof}
Given an arbitrary $p\in S^n$, using that $W=w-\langle w,p\rangle p$ we have
\begin{align*}
|W_i|^2&=|w_i|^2-2\langle w_i,p\rangle^2+\langle w,p\rangle^2|p|^2\\&=1-\langle
w,p\rangle^2.
\end{align*}
Hence
\begin{align*}
\sum_{i=0}^n|W_i|^2=(n+1)-\sum\langle w_i,p\rangle^2=n+1-|p|^2=n,
\end{align*}
giving the first claim. Then second claim is a similar calculation, noting that
if $v$ is in the tangent space at $p$ then
\begin{align*}
\langle W_i,v\rangle=\langle w_i-\langle w_i,p\rangle p,v\rangle=\langle
w_i,v\rangle.
\end{align*}
The result follows.
\end{proof}
\end{lemma}

\begin{lemma} \label{confseclemma10} Let $W$ be a conformal vector field on the
sphere corresponding to
a choice $w\in\mathbb{R}^{n+1}$, and let $\omega(p) := \langle w,p\rangle$.
Then for any $f:M\to S^n$ we have
\begin{enumerate}
\item $\N_{\ga} W = - \gw Tf_{\ga}$.
\item $\lap W = - \langle W, Tf_{\ga} \rangle Tf_{\ga} - \gw \tau $.
\end{enumerate}
\begin{proof}
Let $\rN$ denote the Levi-Civita connection on $S^n$, and $\bN$ denote the
Levi-Civita connection on $\mathbb{R}^{n+1}$. The undecorated connection will be
the induced connection on $TM \ten f^* (TN)$. Then we have that $\rN (\cdot) =
\prs{ \bN (\cdot)}^T$, where the $T$ superscript indicates projection onto the
tangent space of the sphere. We compute
\begin{align*}
\N_{\ga} W
&= f^{i}_{\ga}  \rN_i W \\
&= f^{i}_{\ga} \left( \bN_i W \right)^{T} \\
&= f^{i}_{\ga} \left( \bN_i \lb V^j \del_j - V_k x^k x^j \del_j  \rb \right)^{T}
\\
&= f^{i}_{\ga} \left(  (- V_k \delta_i^k x^j  - V_k x^k \delta_i^j ) \del_j
\right)^{T} \\
&= f^{i}_{\ga}  \left(  - V_i x^j \del_j   - V_k x^k \del_i \right)^{T} \\
&= f^{i}_{\ga} \left( - V_k x^k \del_i \right) \\
&=- \gw f^i_{\ga} \del_i.
\end{align*}
This proves the first claim. Next we have that
\begin{align*}
\N_{\ga} \N_{\ga} W
&= - \N_{\ga} \lb \gw Tf_{\ga} \rb \\
&= -( \N_{\ga} \gw ) Tf_{\ga} - \gw ( \N_{\ga} Tf_{\ga}) \\
&= -( (\del_{\ga} f^{i}) \rN_{i} \gw ) Tf_{\ga} - \gw \tau \\
&= - (\del_{\ga} f^{i})  (\bN_{i} (V^k x^k))^T Tf_{\ga} - \gw \tau \\
&= - (\del_{\ga} f^{i})   (V_i)^T Tf_{\ga} - \gw \tau \\
&= - (\del_{\ga} f^{i})  W_i Tf_{\ga} - \gw \tau \\
&= - \langle W, Tf_{\ga} \rangle Tf_{\ga} - \gw \tau.
\end{align*}
The result follows.
\end{proof}
\end{lemma}

\begin{prop} \label{confsecprop20} Let $W$ be a conformal vector field on the
sphere $S^n$ and let
$f:\mathbb{R}^m\to S^n$ be a soliton.  Then
\begin{equation*}
L^f W = 2\langle W,Tf_\alpha\rangle Tf_\alpha-|Tf|^2 W,
\end{equation*}
and consequently
\begin{align*}
\int_{\bRm} \ip{ L^fW, W } G_0 dV & =  \int_{\bRm} -|Tf|^2 |W|^2 + 2 |
\langle W, Tf \rangle |^2 G_0 dV.
\end{align*}

\begin{proof}
Recall that $L^f W=-\Delta W-R(W,Tf_\alpha)Tf_\alpha+\frac{1}{2}\N_xW$.  Using
Lemma \ref{confseclemma10} we have
\begin{align*}
-\Delta W=\langle W,Tf_\alpha\rangle Tf_\alpha+\omega\tau.
\end{align*} 
Next, using the well-known formula for the curvature of the sphere we have
\begin{align*}
-R(W, Tf_{\ga}) Tf_{\ga}
&= -|Tf|^2 W + \langle W , Tf_{\ga} \rangle Tf_{\ga}.
\end{align*}
Lastly, using Lemma \ref{confseclemma10} again yields $\frac{x^{\ga}}{2}
\N_{\ga} W=-
\gw \left( \frac{x}{2} \right) \hook Tf $.  But by the
soliton identity this is $-\omega\tau$. Adding these gives the result.
\end{proof}
\end{prop}

\begin{cor}\label{prop:solitonintospheres}
Let $f:\mathbb{R}^m\to S^n$ be a soliton. If there is a conformal vector field
which is perpendicular to the image of $Tf$ at each point then $f$ is either
unstable or constant.

\begin{proof}
Let $W$ be such a vector field. Then by Proposition \ref{confsecprop20} we have
\begin{align*}
\int_{\bRm}\langle L^fW,W\rangle G_0 dV=-\int_{\bRm}|W|^2|Tf|^2G_0 dV\leq 0,
\end{align*}
since by assumption $\langle W,Tf\rangle=0$. Now if $V$ is the pushforward of
any vector field from $\mathbb{R}^m$, then $\langle V,W\rangle=0$. In
particular, $W$ is orthogonal to the eigenspaces associated to the pushforward
of constant vector fields and the pushforward of the position vector field. If
$f$ is stable these are the only negative eigendirections, but on the other hand
if $f$ is not constant we would have strict negativity of $\int\langle
L^fW,W\rangle G_0 dV$.
\end{proof}
\end{cor}
\begin{cor}\label{prop:conformsum}
Let $w_i$, $0\leq i\leq n$ be an orthonormal basis of $\mathbb{R}^{n+1}$ and let
$W_i$ be the corresponding conformal vector fields on $S^n$. Then
\begin{equation*}
\sum_{i=0}^n\int_{\bRm} \ip{ L^fW_i,W_i } G_0 dV=(2-n)\int_{\bRm}\brs{Tf}^2 G_0
dV.
\end{equation*}
\begin{proof}
This follows by combining Lemma \ref{lemma:confsum} and Proposition
\ref{confsecprop20}.
\end{proof}
\end{cor}
We see from this last corollary that good lower bounds on the spectrum of
$L^f$, which as we have seen is intimately related to the notion of
$\mathcal{F}$-stability, gives an upper bound on the entropy
$\int|Tf|^2G_0dV$. Specifically, we are now equipped to prove the rough upper
bound on the entropy given in Theorem \ref{thm:main2}.
\begin{proof}[Proof of Theorem \ref{thm:main2}]
Let $f:\bRm\to S^n$, $n\geq 3$, be a stable soliton. Let $W_i$ be any basis of
conformal vector fields as before. By Theorem \ref{thm:rayleigh} we have
$\int\langle
LW_i,W_i\rangle G_0dV\geq -\frac{3}{2}\int|W_i|^2G_0dV$, while by
\ref{prop:conformsum} we get
\begin{align*}
-\frac{3n}{2}=-\frac{3}{2}\int_{\bRm}\sum_{i=0}^n|W_i|^2G_0dV\leq
(2-n)\int_{\bRm}|Tf|^2G_0dV.
\end{align*}
The result follows.
\end{proof}

\begin{rmk} It is shown in \cite{Struwe1}, that small entropy suffices to
establish regularity of harmonic map heat flow.  This result shows that, for
maps into spheres, stable singularities cannot have entropy which is too large.
\end{rmk}

It is an open problem to determine exactly when a homotopy class of maps
$f:S^m\to S^n$, $m\geq 3$ contains a harmonic representative. When $m=n\leq 7$
or $m=n=9$ this problem is answered in the affirmative through a harmonic join
construction, as explained in Chapter 9 of \cite{ER}.  It may be the case that
for sufficiently high degree maps from $S^n \to S^n$, the entropy as well as the
entropy of any blowup model, is large.  If so Theorem \ref{thm:main2} would
imply that all such singularities are unstable, and one could then use a generic
harmonic map heat flow to construct harmonic maps in these homotopy classes.

\section{Dimensional Augmentation and Stability} \label{s:dimaug}

The idea of entropy stability aims to allow perturbations of a given map to move
past all but the ``essential,'' stable singularities of a given harmonic map
heat flow.  In this section we extend this strategy by allowing a further
embedding into a higher dimensional image before perturbing.  As we will see
this is particularly well-suited to the case of the sphere, and motivates
certain conjectures on the structure of solitons mapping into spheres.

\subsection{Convex domains and harmonic map flow}\label{ss:existconvVn}

In this subsection we record some notions of convexity and discuss their
relationship to harmonic maps, and the singularity formation of harmonic map
flow, in preparation for Theorem \ref{ltethm}.

\begin{defn}
Let $(M,g)$ be a smooth Riemannian manifold. Given any open set $U \subset M$, a
function $F \in C^2(U)$ is \emph{convex} if the Hessian of $F$ is nonnegative on
$U$.  This convexity is \emph{strict} if nonnegativity is replaced with strict
positivity.  Given $(M, \del M, g)$ a manifold with
boundary, we say that a function $F : M \to \mathbb R$ is \emph{boundary
defining}
if it is strictly convex on $M$ and there exists $c \in \mathbb R$ such that
$M = F^{-1}(-\infty, c]$ and $\del M = F^{-1}(c)$.
\end{defn}

\begin{defn} A set $U \subset M$ is \emph{convex supporting} if
for all $K \subset U$ compact, $K$ admits a strictly convex function $F_K$. Note
that these functions do note have to be defined on all of $U$. A \emph{maximal
open convex supporting set} is one which is not properly contained in any other
open convex supporting set.
\end{defn}

It is important to note that the notion of a convex supporting set is strictly
more general that that of a domain admitting a convex function.  An example of
such a maximal open convex supporting set is constructed in (\cite{JXY},
pp.5-9), and we review this construction here for convenience.

\begin{prop}[\S 2.2 and Theorem 2.1 of \cite{JXY}, pp.6-9]\label{prop:} 
Given $(S^n, g) \subset \mathbb R^{n+1}$, let 
\begin{equation*}
H_+^n := \{ (x_i)_{i=1}^{n+1} \in S^n : x_1 = 0, x_2 \geq 0 \}.
\end{equation*}
Then $S^n \setminus H_+^n$ is a maximal open convex supporting
subset of $S^n$.
\begin{proof}[Sketch of proof] First we show that $S^n \setminus H_+^n$ is a
convex supporting
subset of of $S^n$, and address the maximality later.  We will take an arbitrary
compact set $K\subset \prs{S^n \setminus H_+^n}$ and displaying a convex
function over $K$. To do so consider the projection
\begin{align*}
\pi 
&: S^n \to \bar{D}^2 \\
&: (x_i)_{i=1}^{n+1} \mapsto (x_1,x_2).
\end{align*}
Where $D^2$ is the $2$-disk and the overline indicates taking its closure.
Observe that by construction, $x \in \prs{S^n \setminus H_+^n}$ if and only if
$\pi(x)$ is contained in $D^2$ with the radius connecting the origin to $(0,1)$.
We use this projection and consider the families of equivalence classes in terms
of spherical coordinates, with 
\begin{equation*}
\pi(x) = \prs{\nu_x \sin \varphi_x , \nu_x \cos \varphi_x}.
\end{equation*}
Over $K$, we have that there is some constant $c \in (0,1)$ so that $c > \nu_x$
for all $x \in K$. Then if we choose $\kappa \in \mathbb{R}_{>0}$ sufficiently
large, the function
\begin{equation*}
F_K(x) := \kappa^{-1} \exp \brk{ \kappa \prs{\varphi_x + \arcsin \prs{
\frac{c}{\nu_x}}}}.
\end{equation*}
is strictly convex over $K$.

Now we show that $S^n \setminus H_+^n$ is maximal.  Assume that there is some
larger open convex
supporting set, $W^n$ containing $S^n \setminus H_+^n$. This implies that there
is some $\theta
\in \left(0,\tfrac{\pi}{2} \right]$ so that the point $(0, \sin \theta, y \cos
\theta) \in S^1 \times S^{n-1}$ resides in $W^n$.  Observe that the closed
geodesic
\begin{align*}
\gamma_t
&: \mathbb{R} \to S^1 \times S^{n-1}\\
\gamma_t & := \prs{\sin t, \cos t \sin \theta, y \cos t \cos \theta}
\end{align*}
lies in $W^n$ and has unit speed. We can take an open neighborhood $U$ about
$\gamma(\mathbb{R})$ and construct a strictly convex function $F_U$ over $U$ as
above.
\begin{equation*}
\tfrac{d^2}{dt^2} \brk{F_U \circ \gamma_t} = \N^2 F_U\prs{\tfrac{\del
\gamma_t}{\del t},
\tfrac{\del \gamma_t}{\del t} } > 0.
\end{equation*}
But since $\prs{F_U \circ \gamma_t}$ is periodic it achieves its maximum, a
contradiction.
\end{proof}
\end{prop}

By an elementary argument, Gordon shows \cite{Gordon} that smooth harmonic maps
into convex supporting domains are constant.  We adapt this argument to show
that solitons mapping into convex
supporting domains are also constant.

\begin{prop} Any soliton mapping into a convex supporting domain is constant.
\begin{proof}  We give two proofs.  On the one hand, a soliton satisfies the
requirements of an $n$-obstruction.  But the proof of Theorem \ref{thm:5.4.3LW}
shows that $n$-obstructions do not exist.  On the other hand we can give a
direct proof inspired by the entropy functional.  A direct calculation 
shows that if $F$ is the given convex function, one has
\begin{align*}
\Delta(F\circ f)\geq C|Tf|^2+\ip{ \N F,\frac{x}{2} \hook Tf}.
\end{align*}
Now fix some $R > 0$, let $\eta$ denote a cutoff function for $B_R$ and then
integrate by parts against a Greens function, and estimate with the
Cauchy-Schwarz inequality to yield
\begin{align*}
 \int_{B_R} \brs{Tf}^2 \eta^2 G dV \leq&\ \int_{B_R} \left[ \Delta (F
\circ f) - \ip{ \N F, \frac{x}{2} \hook Tf } \right] \eta G
dV\\
=&\ \int_{B_R} \langle \N (F \circ f), \N \eta \rangle \eta G dV\\
\leq&\ C \int_{B_R \backslash B_{\frac{R}{2}}} \brs{Tf} \brs{\N \eta} \eta
G dV\\
\leq&\ C \ge \int_{B_R \backslash B_{\frac{R}{2}}} \brs{Tf}^2 \eta^2 G dV
+ \frac{C}{\ge} \int_{B_R \backslash B_{\frac{R}{2}}} \brs{\N \eta}^2 G
dV\\
\leq&\ \frac{1}{2} \int_{B_R} \brs{Tf}^2 \eta^2 G dV + C R^{n-2} e^{- R^2}.
\end{align*}
Rearranging this and choosing $R$ sufficiently large shows that $Tf = 0$, as
required.
\end{proof}
\end{prop}

Given this, a natural question to
ask is if the harmonic map flow admits a smooth long time
solution with convergence to a point given an initial mapping into a
convex supporting region.  Our next proposition shows that this is not
the case, and is the reason we are forced to make stronger convexity hypotheses
in showing Theorem \ref{ltethm}.

\begin{figure}[ht]
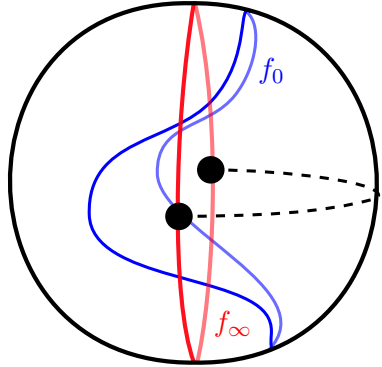
\caption{Harmonic map heat flow on sphere which exits the region $S^2\backslash
H_+^2$}
\label{fig:sphere}
\end{figure}

\begin{prop} There exists a map $f : S^1 \to S^2 \backslash H_+^2$ such that the
harmonic map heat flow with initial condition $f$ exists smoothly for all time
and converges to the great circle $\{x_3 = 0 \}$. Therefore this solution will
leave, in finite time, the domain of definition of any convex function defined
on the image of $f$.
\begin{proof} Figure \ref{fig:sphere} summarizes the situation.  Let
$x_1:=\cos\theta\sin\phi$, $x_2:=\sin\theta\sin\phi$, and $x_3:=\cos\phi$ be
spherical coordinates. Consider the hinge $H_+^2$ where $x_2=0$ and $x_3\geq 0$.
Let $\epsilon>0$ and consider the closed curve in $S^2\setminus H_+^2$ given by
\begin{align*}
f(s)=(\theta(s),\phi(s)) = \prs{s,\tfrac{\pi}{2}+\epsilon\cos(2s)}.
\end{align*}
This curve is invariant (up to unit speed re-parameterization) under the
isometry given by reflection in the $x_3=0$ plane followed by a rotation of
$\frac{\pi}{2}$ radians around the $x_3$-axis.  By uniqueness of solutions to
the flow, this symmetry is preserved along the solution to the flow beginning
from $f$.  Since the domain has dimension $1$, we know that the flow exists for
all time and converges to a closed geodesic.  As the symmetry is preserved in
the limit, the limit must be the great circle with $\{x_3 = 0 \}$.
\end{proof}
\end{prop}

\begin{rmk} It seems quite likely that by allowing initial data which ``bends''
further around the set $H_+^2$ that the flow can exit the entire convex
supporting domain in finite time, instead of infinite time. 
\end{rmk}

\subsection{Proof of Theorem \ref{ltethm}}\label{ss:extESthm}

In this subsection we prove Theorem \ref{ltethm}, which relies strongly on
results from \cite{LW}.  To begin we provide a sketch of the discussion in \S
5.4 of (\cite{LW}, pp.118-123) which discusses a generalization of the classic
theorem of Eells-Sampson \cite{ES}, where the nonpositivity assumption on the
sectional curvature of $N$ is replaced with the weaker assumption that the
universal cover of the target $N$ admits a strictly convex function which grows
quadratically.  We include sketches of some proofs for convenience as we will
reference and exploit some of the arguments as well as the results themselves. 
First we give a preliminary definition, which will ultimately be key in tying in
this work with that of the discussion of solitons.
\begin{defn}\label{defn:nobstr} A smooth map $v_t \in
C^{\infty}\prs{\bRn \times \mathbb{R}_{\leq 0}, N}$ is an \emph{$n$-obstruction}
for $(N,h)$ if it satisfies the following
\begin{enumerate}
\item[(nO1)] $v_t$ satisfies \eqref{eq:HMHF},
\item[(nO2)] $\brs{ \prs{T v_0} (0) } > 0$, and for all $(x,t) \in \bRm \times
\mathbb{R}_{\leq 0}$.
\begin{equation*}
\brs{ \prs{T v_t}(x)} \leq \brs{ \prs{T v_0}(0)},
\end{equation*}
\item[(nO3)] There exists some $E_0 \in \mathbb{R}_{>0}$ such that for all $R >
0$ and $t \in \mathbb{R}_{\leq 0}$,
\begin{equation}\label{eq:nobstr}
R^{2-m} \int_{B_R} \brs{ T v_t }^2 \leq E_0.
\end{equation}
\end{enumerate}
\end{defn}

The next result shows that $n$-obstructions are the obstruction to smooth
existence of the
harmonic map heat flow.

\begin{prop}[\cite{LW} Proposition 5.4.2]\label{prop:5.4.2LW} Assume $(N,h)$
does not admit an $n$-obstruction. Given initial data \eqref{eq:HMHF} has a
unique smooth solution $f_t \in C^{\infty}(M \times \mathbb{R}_{\geq 0},N)$.
Moreover, there exists $C \in \mathbb{R}_{\geq 0}$ such that for all $t \geq 1$,
\begin{equation}\label{eq:5.4.2LW}
\brs{\brs{T f_t}}_{C^0(M)} \leq C.
\end{equation}
\begin{proof}[Sketch of proof]
The proof consists of constructing a blowup limit around a potential singularity
and showing it is an $n$-obstruction.  If the singularity occurs for time $T <
\infty$, one and constructs an $n$-obstruction from the blowup function obtained
by dilating and shifting coordinates, with some sequence of scalars $\kappa_i
\searrow 0$
\begin{equation*}
f_t^i(x) := f_{t_0 + \kappa_i^2 t}(x_0 + \kappa_i x).
\end{equation*}
This immediately satisfies (nO1), and higher order regularity of parabolic
equations yields (nO2). Property (nO3) is verified by comparing the left side of
\eqref{eq:nobstr} to a scaled version of $\mathcal{F}_{x_0,t_0}\prs{f_t}$ and
then using the monotonicity given by Lemma \ref{thm:Struwemonotonicity}.  In the
case of a singularity at time infinity, we can choose a
sequence of times $t_j \nearrow \infty$ and repeat the argument above.
\end{proof}
\end{prop}

\begin{thm}[Theorem 5.4.3 of \cite{LW}, pp.120]\label{thm:5.4.3LW} Let
$(\widetilde{N}, \widetilde{h})$ be the universal covering of $(N,h)$. Suppose
that $\widetilde{N}$ admits a strictly convex function $F \in
C^2\prs{ \widetilde{N}}$ with quadratic growth, i.e. there are positive
$c_0,c_1,c_2 \in \mathbb{R}_{\geq 0}$ such that
\begin{align}
&\N^2 F \geq c_0 \tilde{h} \text{ on $\widetilde{N}$},\label{eq:5.4.3LWa} \\
&0 \leq F(y) \leq c_1 (d_{\widetilde{h}}(y,y_0))^2 + c_2 \text{ for all $y \in
\widetilde{N}$},\label{eq:5.4.3LWb}
\end{align}
for some $y_0 \in \widetilde{N}$, where $d_{\widetilde{h}}$ is the distance
function on $(\widetilde{N}, \widetilde{h})$.  Then given $f : (M^m, g) \to
(N^n,
h)$ there exists a smooth
solution $f_t \in C^{\infty}(M \times \mathbb{R}_{\geq 0}, N)$ to
\eqref{eq:HMHF} satisfying
\eqref{eq:5.4.2LW}. And for suitable $t \nearrow \infty$, $f_t$ converges to a
smooth harmonic map $f_{\infty} : M \to N$ in $C^2(M,N)$.

\begin{proof}[Sketch of proof] It amounts to appealing to Proposition
\ref{prop:5.4.2LW} by demonstrating that $(N,h)$ does not admit any
$n$-obstructions. Since the fundamental group $\pi_1(N)$ acts isometrically on
$\widetilde{N}$ by deck transformation, so given any $F$ satisfying
\eqref{eq:5.4.3LWa} and \eqref{eq:5.4.3LWb}, we have that for all $\ga \in
\pi_1(N)$ the function $F_{\ga} := (F \circ \ga)$ also satisfies
\eqref{eq:5.4.3LWa} and \eqref{eq:5.4.3LWb} (with $y_0$ replaced by
$\ga^{-1}(y_0)$).  If $N$ supports an $n$-obstruction $\nu_t$, then its lift
$\tilde{\nu}_t$ is also an $n$-obstruction on $(\widetilde{N}, \widetilde{h})$.
For any $\ga \in \pi_1(M)$ we consider and application of the heat operator to
$\gw_t = F_{\ga} \circ \tilde{\nu}_t$ and obtain, using \eqref{eq:5.4.3LWa},
\begin{equation*}
(\del_t - \lap) \gw_t = - c_0 \brs{ T \nu_t }^2.
\end{equation*}
Using the representation formula for the heat equation applied to $\gw$ yields,
for $t_0 < 0$,
\begin{align*}
(4 \pi)^{m/2} \gw_0(0)
& \leq \brs{t_0}^{-m/2} \int_{\mathbb{R}^m} \gw_{t_0}(x) e^{\frac{- \brs{x}^2}{4
\brs{t_0}}} dV - \int_{t_0}^0 \int_{\bRm} \brs{T v_t}^2 \brs{t}^{-m/2}
e^{-\tfrac{\prs{x}^2}{4 \brs{t}}}dV dt.
\end{align*}
Since $w = F_{\ga} (\tilde{v}) \geq 0$ by \eqref{eq:5.4.3LWb} then we have that
\begin{equation}\label{eq:intvleqintw}
\int_{t_0}^0 \int_{\mathbb{R}^m} \brs{T \nu_t}^2 \brs{t}^{-m/2} e^{-
\tfrac{\brs{x}^2}{4 \brs{t}}} dV dt \leq c_0^{-1} \brs{t_0}^{-m/2}
\int_{\mathbb{R}^m} \gw_{t_0}(x) e^{- \tfrac{\brs{x}^2}{4 \brs{t_0}}} dV dt.
\end{equation}
There is a choice of $t_0 \in [-2,-1]$ and $\ga \in \pi_1(N)$, verified in
\cite{LW} pp.121-122 (we refer the reader to this verification) so that
\begin{equation}\label{eq:intwc3}
\brs{t_0}^{-m/2} \int_{\bRm} \gw_{t_0}(x) e^{-\tfrac{\brs{x}^2}{2 \brs{t_0}}}
dV <c_3,
\end{equation}
where $c_3 := c_3\prs{ c_0,c_1,c_2, E_0 } \in \mathbb{R}_{>0}$.  If we take the
combination of \eqref{eq:intvleqintw} and \eqref{eq:intwc3} and then apply a
`layer cake' decomposition of the time domain we obtain 
\begin{align*}
\int_{t_0}^0 \int_{\bRm} \brs{T \nu_t}^2 \brs{t}^{-m/2} e^{-\frac{\brs{x}^2}{4
\brs{t}}}dV dt & < \sum_{k= 1}^{\infty} \int_{-4^{-k}}^{-4^{-(k+1)}}
\int_{\mathbb{R}^m} \brs{T \nu_t}^2 G_{0} dV dt\\
& \leq c_0^{-1} c_3 \\
&= c_4 < + \infty.
\end{align*}
This implies that for a given $\delta>0$ there is some index $k_{\delta}$
within the summation over which 
\begin{equation*}
I_{k_{\delta}} := \int_{4^{-k_{\delta}}}^{4^{-(k_{\delta}+1)}}
\int_{\mathbb{R}^m} \brs{T \nu_t}^2 G_0 dV dt \leq \delta^2.
\end{equation*}
We may choose $\delta \in \mathbb{R}_{>0}$ sufficiently small coming from
Proposition 7.1.4 in \cite{LW} which yields that one has $\brs{T \nu_0} (0)
\leq c_5$ for some $c_5 = c_5(\delta, n, \tilde{N}, E_0)$.

Conversely, for every $\kappa \in \mathbb{R}_{>0}$ we can construct another
$n$-obstruction map via rescaling with $\nu^{\kappa}_t(x) = \nu_{\kappa^2
t}(\kappa x)$.
We can apply, for each one of these $n$-obstruction tensors, the same argument.
Furthermore if we take a limit over $\kappa \to \infty$ their limit will also be
an
$n$-obstruction tensor. But then we have
\begin{equation*}
\lim_{\kappa \to \infty} \brs{T \nu^{\kappa}_0}(0) = \lim_{\kappa \to \infty}
\kappa
\brs{T \nu_0}(0) = \infty.
\end{equation*}
So the limit cannot be an $n$-obstruction because (nO2) of the definition is
violated. We conclude an $n$-obstruction tensor cannot exist on $N$, and the
result follows.
\end{proof}
\end{thm}

With these preliminaries in place we can establish Theorem \ref{ltethm}.

\begin{proof}[Proof of Theorem \ref{ltethm}] We consider the
solution $f_t$ to harmonic map flow with initial condition $f : M \to
\overset{\circ}{N}$.  Let $\bar{N}$ denote the doubling of $N$,
i.e. $\bar{N} = N \sqcup_{\del N} N$.  Using cutoff functions in a small
neighborhood of the boundary we can double the metric $h$ as well.  We can
interpret the solution to the harmonic map flow with initial condition $f$ as a
flow into $\bar{N}$, and then we know by Proposition \ref{prop:5.4.2LW} that if
the flow encounters a singularity then there exists an $n$-obstruction mapping
into $\bar{N}$.

We further claim that for all $t \geq 0$,
$f_t(M) \subset \overset{\circ}{N}$ by a strong maximum principle argument.  Let
$F$ be the strictly convex
boundary-defining function on $(N, \del N, h)$, where $N = F^{-1}(-\infty,c]$
and $\del N = F^{-1}(c)$.  Note that there exists some $\ge > 0$ such that $F
\circ f_0 \leq c - \ge$.  Moreover, since $N$ is compact, $F^{-1}(-\infty,c -
\ge]$ is also compact, and hence the only way the flow can leave this set is by
increasing $F$ beyond the threshold of $c - \ge$.

So, let $\Phi_t = F
\circ f_t$.  A direct computation shows that 
\begin{align*}
\dt \brk{\Phi_t} =&\ \gD \Phi_t- \N^2 F(Tf_\alpha, Tf_\alpha)\leq \gD \Phi_t.
\end{align*}
It follows that $\sup_M \Phi_t \leq \sup_M \Phi_0$ for all $t \geq 0$, and hence
$f_t \subset F^{-1}(-\infty,c - \ge] \subset \overset{\circ}{N}$.  Thus, the
construction of Proposition \ref{prop:5.4.2LW} in fact generates an
$n$-obstruction into $\overset{\circ}{N}$, which admits a strictly
convex function.  This is a contradiction as outlined in the proof of Theorem
\ref{thm:5.4.3LW}.  Hence the solution exists for all time with a uniform
estimate on the energy density.  This implies subconvergence to a limiting
harmonic map, which is necessarily a constant map by (\cite{Gordon}).
\end{proof}

\subsection{Conjectural framework for stable singularities into spheres}
\label{ss:conjs}

Let us now discuss the strategy of dimensional augmentation and how it relates
to singularity formation for the harmonic map heat flow.  Consider a harmonic
map heat flow $f_t : M \to S^n$ which encounters a singularity at $T < \infty$. 
We can use a hemispherical embedding $S^n \subset S^{n+1}$ and think of the
solution as a family $\til{f}_t : M \to S^{n+1}$.

We have already exhibited
several ways in which the singularity encountered by $\til{f}$ is unstable.  In
particular 
Corollary \ref{prop:solitonintospheres} shows that any soliton singularity model
for $\til{f}$ will be unstable or constant.  Alternatively, one can use Theorem
\ref{ltethm} to exhibit the instability.  In particular, for any time $t < T$,
one can perturb the map $\til{f}_t$ into the upper $x_{n+2}$ hemisphere via
\begin{align*}
 \bar{f}_{\ge} = \left( \sqrt{1 - \ge^2} f^1, \dots, \sqrt{1-\ge^2} f^{n+1}, \ge
\right),
\end{align*}
where we have used the explicit components of the given map $f$.  The image lies
entirely in a sublevel set of the strictly convex function $x_{n+2}$, and hence
Theorem \ref{ltethm} guarantees the long time existence and convergence of the
flow with initial condition $\til{f}_{\ge}$.

This gives an alternative way of moving past even a stable singularity of the
flow mapping into spheres.  In particular, one could choose a sequence of times
$t_i \to T$ and a sequence $\ge_i \to 0$, and then consider the long time
solutions to the flow with initial condition $\bar{f}_{\ge_i}(t_i)$, and
consider
the limit as $i$ goes to infinity. A similar strategy was employed by Grayson
and Altschuler \cite{GA} for the curve shortening flow, following a suggestion
of Calabi as follows. Consider a solution $\gamma_t(s)$ to the curve shortening
flow in the plane which has a singularity at $t=T<\infty$. Consider so called
ramps over $\gamma$, that is embedded curves in $\mathbb{R}^3$ with
$\frac{d\gamma_3}{ds}>0$ and which project to parameterizations of $\gamma$. One
then takes a sequence of ramps which approximate $\gamma$ and for which the
curve shortening flow for this sequence exists on a uniformly longer time
interval, and then takes the approximation parameter to zero. Once the right
estimates are obtained, the limiting flow exists for all time and agrees with
the evolution of $\gamma$ away from singularities. Constructing a ``flow past
singularities'' for the harmonic map heat flow may be approachable using similar
methods.

\bibliographystyle{hamsplain}

\end{document}